\documentclass[11pt, leqno]{amsart}

\usepackage{verbatim}
\usepackage{etex}
\usepackage{pictexwd}
\usepackage{lscape}
\usepackage{graphicx}
\usepackage{amssymb,amsmath,amsthm,amscd, amsfonts}
\usepackage{mathrsfs}
\usepackage{enumerate}
\usepackage[usenames,dvipsnames]{color}
\usepackage[vcentermath,enableskew,noautoscale]{youngtab}
\usepackage[all]{xy}

\allowdisplaybreaks[4]
\newcommand{\nc}{\newcommand}
\numberwithin{equation}{section}

\newenvironment{red}{\relax\color{red}}{\relax}
\newenvironment{blue}{\relax\color{blue}}{\hspace*{.5ex}\relax}
\newenvironment{jaune}{\relax\color{Orchid}}{\hspace*{.5ex}\relax}

\newcommand{\bj}{\begin{jaune}}
\newcommand{\ej}{\end{jaune}}

\newcommand{\ber}{\begin{red}}
\newcommand{\er}{\end{red}}
\newcommand{\bebl}{\begin{blue}}
\newcommand{\ebl}{\end{blue}}

\theoremstyle{plain}
\newtheorem{lemma}{Lemma}[section]
\newtheorem{prop}[lemma]{Proposition}
\newtheorem{theorem}[lemma]{Theorem}

\newcommand{\Prop}{\begin{prop}}
\newcommand{\enprop}{\end{prop}}
\newcommand{\Lemma}{\begin{lemma}}
\newcommand{\enlemma}{\end{lemma}}
\newcommand{\Th}{\begin{theorem}}
\newcommand{\enth}{\end{theorem}}
\newtheorem{corollary}[lemma]{Corollary}
\newcommand{\Cor}{\begin{corollary}}
\newcommand{\encor}{\end{corollary}}
\newtheorem{definition}[lemma]{Definition}
\newtheorem{conjecture}[lemma]{Conjecture}
\newcommand{\Def}{\begin{definition}}
\newcommand{\edf}{\end{definition}}
\newtheorem{sublemma}[lemma]{Sublemma}
\newcommand{\Sublemma}{\begin{sublemma}}
\newcommand{\ensub}{\end{sublemma}}

\theoremstyle{definition}
\newtheorem{remark}[lemma]{Remark}
\newtheorem{example}[lemma]{Example}
\newtheorem{Convention}[lemma]{Convention}
\newcommand{\Conv}{\begin{Convention}}
\newcommand{\enconv}{\end{Convention}}
\nc{\Con}{\begin{conjecture}}
\nc{\encon}{\end{conjecture}}
\nc{\Rem}{\begin{remark}}
\nc{\enrem}{\end{remark}}
\newcommand{\C}{{\mathbb C}}

\newcommand{\Z}{{\mathbb Z}}
\newcommand{\B}{{\mathbf{B}}}
\newcommand{\A}{{\mathbf A}}

\newcommand{\one}{{\bf{1}}}

\newcommand{\g}{{\mathfrak{g}}}

\newcommand{\Hom}{\operatorname{Hom}}
\newcommand{\End}{\operatorname{End}}

\newcommand{\Ext}{\operatorname{Ext}}

\newcommand{\eq}{\begin{eqnarray}}
\newcommand{\eneq}{\end{eqnarray}}

\newcommand{\eqn}{\begin{eqnarray*}}
\newcommand{\eneqn}{\end{eqnarray*}}
\newcommand{\on}{\operatorname}

\newcommand{\QED}{\end{proof}}
\newcommand{\Proof}{\begin{proof}}

\newcommand{\soplus}{\mathop{\mbox{\normalsize$\bigoplus$}}\limits}

\newcommand{\ba}{\begin{array}}
\newcommand{\ea}{\end{array}}

\newcommand{\set}[2]{\left\{#1 \mid #2 \right\}}

\newcommand{\eqsub}{\begin{subequations}\begin{eqnarray}}
\newcommand{\eneqsub}{\end{eqnarray}\end{subequations}}

\newcommand{\ol}{\overline}

\nc{\la}{\lambda}
\nc{\lam}{\lambda}
\nc{\U}[1][\g]{U_q(#1)}
\nc{\te}{\tilde{e}}
\nc{\tei}{\tilde{e}_i}
\nc{\tf}{\tilde{f}}
\nc{\tfi}{\tilde{f}_i}
\nc{\tU}{\widetilde U_q(\g)}
\nc{\tE}{\tilde{E}}
\nc{\tF}{\widetilde{\F}}
\nc{\tK}{\widetilde{K}}

\nc{\tk}{\tilde{k}}
\nc{\tkone}{\tk_{\ol{1}}}
\nc{\teone}{\tilde{e}_{\ol{1}}}
\nc{\tfone}{\tilde{f}_{\ol{1}}}

\nc{\teibar}{\tilde{e}_{\ol{i}}} \nc{\tfibar}{\tilde{f}_{\ol{i}}}
\nc{\tki}{{\tk}_{\ol {i}}}

\nc{\BZ}{{\mathbb{Z}}}
\nc{\al}{\alpha}
\nc{\qs}{{q}}
\nc{\lan}{\langle}
\nc{\ran}{\rangle}
\nc{\re}{{\mathrm{re}}}
\nc{\wt}{\operatorname{wt}}
\nc{\ch}{\operatorname{ch}}
\nc{\Um}[1][\g]{U^-_q(#1)}
\nc{\Ue}{U^+_q(\g)}
\nc{\eps}{\varepsilon}
\nc{\vphi}{\varphi}
\nc{\sphi}{\varphi^*}
\nc{\seps}{\varepsilon^*}

\nc{\nn}{\nonumber}

\nc{\vp}{\varpi}
\nc{\cls}{{\operatorname{cl}}}
\nc{\Wt}{{\operatorname{Wt}}}
\nc{\Us}{U'_q(\g)}
\nc{\La}{\Lambda}
\nc{\tLa}{\widetilde\Lambda}
\nc{\ro}{{\rm(}}
\nc{\rf}{{\rm)}}
\nc{\norm}{{\mathrm{norm}}}
\nc{\qbox}{\quad\mbox}
\nc{\braid}{{\mathfrak{B}}}
\nc{\Ad}{\operatorname{Ad}}
\nc{\Aut}{\operatorname{Aut}}
\nc{\dt}[1]{\tilde{\tilde #1}}
\nc{\Sn}{S^{{\mathrm{norm}}}}
\nc{\aff}{{\rm{aff}}}
\nc{\rk}{{\mathrm{rk}}}
\nc{\tP}{\widetilde{P}}
\nc{\tW}{\widetilde{W}}
\nc{\Dyn}{\mathrm{Dyn}}
\nc{\tD}{\widetilde{\Delta}}
\nc{\height}[1]{{\operatorname{ht}}(#1)}
\nc{\bl}{\bigl(}
\nc{\br}{\bigr)}
\nc{\Hecke}{\mathrm{H}}
\nc{\HA}{\Hecke^{\mathrm{A}}}
\nc{\HB}{\Hecke^{\mathrm{B}}}
\newcommand{\scbul}{{\,\raise1pt\hbox{$\scriptscriptstyle\bullet$}\,}}
\nc{\vac}{{\phi}}
\nc{\Bt}{\B_\theta(\g)}
\nc{\be}{\begin{enumerate}}
\nc{\ee}{\end{enumerate}}
\nc{\low}{{\mathrm{low}}}
\nc{\upper}{{\mathrm{up}}}
\nc{\Zodd}{\Z_{\mathrm{odd}}}
\nc{\Ft}[1][n]{\mathbb{P}\mathrm{ol}_{#1}}
\nc{\Ftf}[1][n]{\widetilde{\mathbb{P}\mathrm{ol}}_{#1}}
\nc{\KA}{\on{K}^{\mathrm{A}}}
\nc{\KB}{\on{K}^{\mathrm{B}}}
\nc{\Res}{\on{Res}}
\nc{\Fc}[1][{n,m}]{\mathbf{F}_{#1}}
\nc{\tphi}{\tilde{\varphi}}
\nc{\CO}{\mathscr{O}}
\nc{\inte}{\mathrm{int}}
\nc{\Oint}{\mathcal{O}^{\ge0}_{\inte}}
\nc{\vs}{\vspace*}
\nc{\tLt}{\widetilde{L}}
\nc{\tL}{\widetilde{\Lambda}}
\nc{\tu}{\tilde{u}}
\nc{\noi}{\noindent}
\nc{\heigh}{\mathfrak{t}}
\nc{\lowest}{\mathfrak{l}}
\nc{\rootl}{\mathsf{Q}}
\nc{\cl}{\colon}
\nc{\uqpg}{U'_q(\mathfrak g)}
\nc{\uq}{\uqpg}
\nc{\Oh}{\widehat{\mathcal{O}}}
\nc{\pn}{p_{\mathfrak{n}}}

\nc{\KLR}{KLR algebra}
\nc{\KLRs}{KLR algebras}
\nc{\cor}{\mathbf{k}}
\nc{\cora}{{\cor(A)}}
\nc{\haut}{\mathrm{ht}}
\nc{\tens}{\mathop\otimes}
\nc{\gmod}{\mbox{-$\mathrm{gmod}$}}
\nc{\gMod}{\mbox{-$\mathrm{gMod}$}}
\nc{\proj}{\mbox{-$\mathrm{proj}$}}
\nc{\gproj}{\mbox{-$\mathrm{gproj}$}}
\nc{\smod}{\mbox{-$\mathrm{mod}$}}
\nc{\Mod}{\mbox{-$\mathrm{Mod}$}}
\nc{\h}{\mathfrak h}
\nc{\Rnorm}{R^{\rm{norm}}}

\nc{\Vhat}{\widehat{V}}
\nc{\F}{\mathcal{F}}

\def\T{{\mathcal T}}

\nc{\fd}[1][A]{\on{\mathrm{flat.dim}_{#1}}}
\nc{\bP}{{\mathbb{P}}}
\nc{\bPh}{\widehat{\mathbb{P}}}
\nc{\bK}[1][{n}]{\widehat{\mathbb{K}}_{#1}}
\nc{\bV}[1][{n}]{\widehat{V}^{\otimes{#1}}}
\nc{\bVK}[1][{n}]{\widehat{V}^{\otimes{#1}}_{\widehat{\mathbb{K}}}}
\nc{\hV}{\widehat{V}}
\nc{\opp}{\mathrm{opp}}
\nc{\col}{\colon}
\nc{\bnum}{\be[{\rm(i)}]}
\nc{\oep}{\epsilon}
\nc{\qtext}{\quad\text}
\nc{\qtextq}[1]{\quad\text{#1}\quad}
\nc{\longtwoheadrightarrow}[1][]{\xymatrix{\ar@{->>}[r]^-{{#1}}&}}
\nc{\epiTo}[1][]{\longtwoheadrightarrow[{#1}]}
\nc{\epito}{\twoheadrightarrow}
\nc{\monoTo}[1][]{\xymatrix{\ar@{>->}[r]^-{{#1}}&}}
\nc{\sym}{\mathfrak{S}}
\nc{\inp}[1]{{({#1})_{\mathrm{n}}}}
\nc{\rtl}{\rootl}
\nc{\wtd}{\widetilde}
\nc{\etens}{\boxtimes}
\nc{\ds}[1]{\mathrm{d}(#1)}
\nc{\rmat}[1]{{\mathbf{r}}_%
{\mspace{-2mu}\raisebox{-.6ex}{${\scriptstyle{#1}}$}}}
\nc{\rmats}[1]{{\mathbf{r}}_%
{\mspace{-2mu}\raisebox{-.6ex}{${\scriptscriptstyle{#1}}$}}}
\nc{\shc}{\mathcal{C}}
\nc{\shs}{\mathcal{S}}
\nc{\Fct}{{\on{Fct}}}
\nc{\tC}{\widetilde{\shc}}
\nc{\Zp}{\Z_{\ge0}}
\nc{\tPhi}{\widetilde{\Phi}}
\nc{\tT}{{\widetilde{\T}}}
\nc{\Ob}{\on{Ob}}
\nc{\bwr}{\mbox{\large$\wr$}}
\nc{\Img}{\on{Im}}
\nc{\Ab}{\mathcal{A}^{\mathrm{big}}}
\nc{\Sb}{\mathcal{S}^{\mathrm{big}}}
\nc{\As}{\mathcal{A}}
\nc{\Ss}{\mathcal{S}}
\nc{\ntens}{\widetilde{\otimes}}
\nc{\hR}{\widehat{R}}
\nc{\nconv}{\mathop{\mbox{\large $\odot$}}}
\nc{\snconv}{\mbox{\scriptsize$\odot$}}
\nc{\ts}{\tilde{s}}
\nc{\sho}{\mathcal{O}}
\nc{\bc}{\begin{cases}}
\nc{\ec}{\end{cases}}
\nc{\slnh}{{\widehat{\mathfrak{sl}}_N}}
\nc{\UA}{U_q'(\slnh)}
\nc{\KR}{R_K}
\nc{\cQ}{\mathcal{Q}}
\nc{\Irr}{\mathcal{I}rr}
\nc{\tQ}{\widetilde{\cQ}}
\nc{\bs}{\mathbf{s}}
\nc{\bL}{\mathbb{L}}
\nc{\tg}{\tilde{g}}

\nc{\conv}{\mathbin{\mbox{\large $\circ$}}}
\nc{\shconv}{\mathbin{\large\diamond}}
\nc{\hconv}{\mathbin{\mbox{\Large $\shconv$}}}

\nc{\Rm}{R^{\mathrm{ren}}}

\nc{\bQ}{\ol{Q}}

\nc{\de}{\on{\textfrak{d}}}

\nc{\xmono}{\ar@{>->}}
\nc{\xepi}{\ar@{->>}}
\nc{\db}[1]{\raisebox{-.5ex}[2ex][1.8ex]{$#1$}}
\nc{\wb}[1]{\mbox{$\rule[-1.1ex]{0ex}{2ex}#1$}}
\nc{\univ}{\mathrm{univ}}
\nc{\rM}{{}^*\mspace{-2mu}M}
\nc{\lM}{M^*}
\nc{\uqm}{\uq\smod}
\nc{\tR}{\widetilde{R}_{\gamma,\beta}}
\nc{\tx}{\tilde{x}}
\nc{\bi}{\mathbf{i}}
\nc{\ttau}{\widetilde{\tau}}

\nc{\tEnd}{\on{\widetilde{E}nd}}
\nc{\tHom}{\on{\widetilde{H}om}}

\nc{\K}{{J}}
\nc{\Kex}{{\K}_{\mathrm{ex}}}
\nc{\Kfr}{{\K}_{\mathrm{f\mspace{.01mu}r}}}
\nc{\coro}{\cor}
\nc{\tB}{\widetilde{B}}
\nc{\seed}{\mathscr{S}}

\nc{\up}{\mathrm{up}}
\nc{\bfa}{\mathbf{a}}

\newcommand{\cont}{{\rm cont}}

\newlength{\mylength}
\title
{Supersymmetric polynomials and the center of the walled Brauer algebra}

\author[Ji Hye Jung and Myungho Kim]{Ji Hye Jung$^{1}$ and Myungho Kim$^{2}$}

\address{Sogang research team for discrete and geometric structures,
 Department of Mathematics, Sogang university
  Seoul 04107, Korea}
        \email{jung.ji.hye@hotmail.com}

\address{Department of Mathematics, Kyung Hee University, Seoul 02447, Korea}
\email{mkim@khu.ac.kr}

\thanks{$^{1}$
This research was supported by Sogang Research Team for Discrete
and Geometric Structures
and
by NRF-2013R1A1A2063671. } 

\thanks{$^{1,2}$ 
The authors were supported by the European Research Council under the European Union's Framework Programme H2020 with ERC Grant Agreement number 647353 Qaffine}


\date{\today}
\begin{document}

\begin{abstract}
 We study a commuting family of elements of the walled Brauer algebra $B_{r,s}(\delta)$, called the \emph{Jucys-Murphy elements}, and show that the supersymmetric polynomials in these elements belong to the center of the walled Brauer algebra.
When $B_{r,s}(\delta)$ is semisimple, we show that those supersymmetric polynomials generate the center.
Under the same assumption,
we define a maximal commutative subalgebra of $B_{r,s}(\delta)$, called the \emph{Gelfand-Zetlin subalgebra}, and show that it  is generated by the Jucys-Murphy elements.
 As an application, we construct a complete set of primitive orthogonal  idempotents of $B_{r,s}(\delta)$, when it is semisimple.
We also give an alternative proof of a part of the classification theorem of blocks of $B_{r,s}(\delta)$ in non-semisimple cases, which appeared in
the work of Cox-De~Visscher-Doty-Martin.
Finally, we present an analogue of Jucys-Murpy elements for the quantized walled Brauer algebra
$H_{r,s}(q,\rho)$ over $\C(q, \rho)$ and by taking the classical limit we show that the supersymmetric polynomials in these elements generates the center. 
It follows that  H. Morton conjecture, which appeared in the study of the relation between the framed HOMFLY skein on the annulus and that on the rectangle with designated  boundary points, holds if we extend the scalar from  $\Z[q^{\pm1},\rho^{\pm1}]_{(q-q^{-1})}$ to $\C(q, \rho)$.
\end{abstract}

\maketitle


\section*{Introduction}
The \emph{Jucys-Murphy elements} of the group algebra $\C[\sym_r]$ of the symmetric group of $r$ letters are given by
\eqn
 L_k:= \sum_{j=1}^{k-1} (j,k)  \quad (1 \le k \le r),
\eneqn
where $(a,b)$ denotes the transposition exchanging $a$ and $b$ for $1\le a, b \le  r$.
In particular, $L_1=0$ and $L_k$'s are commuting to each other.
These elements were introduced independently in \cite{Jucys, Murphy81} and it was shown that the center of
$\C[\sym_r]$ consists of all the symmetric polynomials in these elements (\cite{Jucys,Murphy}).
This remarkable fact leads various interesting studies.
For example,
 the ring homomorphism from the ring of symmetric polynomials to the center
of $\C[\sym_r]$, which is called the \emph{Jucys-Murphy specialization}, has been studied by many researchers.
Let $f(x_1,\ldots, x_r)$ be a symmetric polynomial.
Since the evaluation $f(L_1, \cdots, L_r)$ belongs to the center of $\C[\sym_r]$,
it can be written uniquely as a linear combination of the natural basis  $\set{ C_{\mu}}{\mu \ \text{is a partition of } r}$ of the center, where $C_{\mu}$ denotes the sum of all
 permutations with the same cycle type $\mu$.
 That is, in the center of $\C[\sym_r]$, we have  an equation
$$f(L_1, \ldots, L_r)=\sum_{\mu} a_{\mu}^{f } C_{\mu} \quad  (a^f_{\mu} \in\C). $$
The problem to calculate coefficients $a^{ f }_{\mu}$ for various symmetric polynomials $f$ is called the \emph{class expansion problem}.
For the elementary symmetric polynomial $e_k$ and the power sum symmetric polynomial $p_k$,
the class expansion problem has been completely solved in \cite{Jucys} and \cite{LascouxThibon}, respectively.
 For the monomial symmetric polynomial $m_{\la}$, a description of
$a^{m_{\la}}_\mu$ for the partitions $\mu$ such that $\ell(\mu) = r-|\la|$
was given in
\cite{MatsumotoNovak}, where $\ell(\mu)$ denotes the number of nonzero parts of $\mu$ and $|\la|$ denotes the sum of  all the parts of $\la$.
In \cite{Lassalle}, Lassalle solved the problem for a large family of symmetric polynomials including complete homogeneous symmetric polynomials $h_k$. Feray reproduced  Lassalle's result in a different way (\cite{Feray}).
Another application of Jucys-Murphy elements is Okounkov and Vershik's
beautiful approach to the representation theory of the symmetric groups (\cite{OV2}).
 Observing simultaneous eigenspaces and eigenvalues of the Jucys-Murphy elements $L_1, \ldots, L_{r}$ in an irreducible $\C[\sym_r]$-module,
 they gave a natural explanation about the appearance of Young diagrams and standard tableaux in the representation theory of symmetric groups.

The \emph{walled Brauer algebra} $B_{r,s}(\delta)$ was introduced independently in \cite{Koike} and \cite{Turaev}
 (See also \cite{BCHLLS}).
  When $\delta=n$, it was studied as the centralizer algebra of the action of general linear Lie algebra $\mathfrak{gl}(n)$ on the mixed tensor space.
Moreover, when $\delta=m-n$,
it is related with the centralizer algebra of the action of general linear Lie superalgebra $\mathfrak{gl}(m,n)$
on the mixed tensor superspace (See, for example, \cite{BS,CW,SM}).
We call them the \emph{mixed Schur-Weyl dualities}.
As the mixed Schur-Weyl duality can be regarded as the generalization of the Schur-Weyl duality,
one can consider the walled Brauer algebra as a natural generalization of the group algebra $\C[\sym_r]$ of the symmetric group $\sym_r$.
Actually, in the diagrammatic description of the walled Brauer algebra  $B_{r,s}(\delta)$, it is easily seen that a copy of the group algebra  $\C[\sym_r]$ of the symmetric group $\sym_r$ is contained in it as a subalgebra.
Thus it is natural to try to find a nice family of elements of $B_{r,s}(\delta)$ containing the Jucys-Murphy elements of $\C[\sym_r]$ and possessing
similar properties with them.

In \cite{BS}, Brundan and Stroppel defined a family of Jucys-Murphy elements $x_1^R,\ldots,x_{r+s}^R$ of $B_{r,s}(\delta)$
 and conjectured that the symmetric polynomials in $x_1^R, \ldots, x_{r+s}^R$ generate the center of $B_{r,s}(\delta)$.
 In \cite{SartoriStroppel}, Sartori and Stroppel worked in more general setting,
 which is called the \emph{walled Brauer category}.
 The category includes the usual walled Brauer algebra $B_{r,s}(\delta)$ as an idempotent truncation.
 If our focus restricts to the case of $B_{r,s}(\delta)$, they defined the Jucys-Murphy elements $\xi_1, \ldots, \xi_{r+s}$ and conjectured that the \emph{doubly symmetric polynomials which satisfy the Q-cancellation property with respect to the $r$-th and $(r+1)$-th variables in $\xi_1, \ldots, \xi_{r+s}$} generate the center of $B_{r,s}(\delta)$.
 A doubly symmetric polynomial which satisfies the above Q-cancellation property
 is also called a \emph{supersymmetric polynomial} in other literatures:
 it is a polynomial in $x_1,\ldots,x_r,y_1,\ldots,y_s$, symmetric in $x_1,\ldots,x_r$ and in
 $y_1,\ldots,y_s$ respectively, and the substitution $x_r=-y_1=t$ yields a polynomial in $x_1,\ldots,x_{r-1}, y_2,\ldots,y_s$, which is independent of $t$.
Note also that in \cite{RuiSu}, Rui and Su introduced a family of elements of $B_{r,s}(\delta)$, called the \emph{Jucys-Murphy-like elements}, in  their study of affine walled Brauer algebras.

 We are strongly motivated by the conjecture of \cite{SartoriStroppel}.
In this paper, we define a  
family of Jucys-Murphy elements $L_1, \ldots, L_{r+s}$ of $B_{r,s}(\delta)$, which  can be regarded as  a modification of the ones in \cite{SartoriStroppel},
and show that supersymmetric polynomials in $L_1, \ldots, L_{r+s}$ is central in $B_{r,s}(\delta)$.
Here, we make use of some relations between generators of $B_{r,s}(\delta)$ and the Jucys-Murphy elements (Proposition \ref{prop:commuting relations with generators}).
Another key ingredient is a theorem by  Stembridge \cite{Stem} saying that the ring of supersymmetric polynomials is generated by the \emph{power sum supersymmetric polynomials}.
 We study the eigenvalues of the supersymmetric polynomials in Jucys-Murphy elements on the cell modules over $B_{r,s}(\delta)$.
Our main theorem is that for the case when $B_{r,s}(\delta)$ is semisimple, which is the case except finitely many values $\delta$, the supersymmetric polynomials in $L_1, \ldots, L_{r+s}$ generate the center of $B_{r,s}(\delta)$ (Theorem \ref{thm:main}).
It follows by a modification of an argument, which was used by Li in \cite{Li} to produce a certain family of symmetric polynomials.
As an application, we can mimic the Okounkov-Vershik's approach to the representation theory of the symmetric groups: when $B_{r,s}(\delta)$ is semisimple, we define the \emph{Gelfand-Zetlin subalgebra} of $B_{r,s}(\delta)$ as the subalgebra generated by centers of certain naturally chosen subalgebras, which are isomorphic to walled Brauer algebras of lower ranks, and show that it is generated by the Jucys-Murphy elements $L_1,\ldots, L_{r+s}$.
Note that the Gelfand-Zetlin subalgebra is a maximal commutative subalgebra of $B_{r,s}(\delta)$, since the branching graph of $B_{r,s}(\delta)$ with respect to our choice of the family of subalgebras is multiplicity-free.
A complete set of primitive orthogonal idempotents of $B_{r,s}(\delta)$ can now be constructed easily.
In addition,
by observing a connection between the eigenvalues of the supersymmetric polynomials in Jucys-Murphy elements  and the conditions in the characterization of blocks of $B_{r,s}(\delta)$ in non-semisimple cases, we can recover a part of the classification theorem of blocks of $B_{r,s}(\delta)$ appeared in \cite{CDDM} (Proposition \ref{prop:blocks}).
This strengthens our belief that the center of $B_{r,s}(\delta)$ is generated by the supersymmetric polynomials in Jucys-Murphy elements, even when $B_{r,s}(\delta)$ is not semisimple
(Conjecture \ref{conj:center}).
Lastly, we consider the case of the \emph{quantized walled Brauer algebras}.
A family of one parameter deformation  of $B_{r,s}(N)$ $(N \in \Z_{\ge 0})$ has been appeared in \cite{KM2} and a two parameter version has been introduced in  \cite{KM1} and  \cite{Leduc}.
 Surprisingly enough, in his study of a connection between the \emph{framed HOMFLY skein module on the annulus} and the one on the rectangle with designated input and output boundary points (\cite{Morton}), 
Hugh Morton conjectured that the center of the quantized walled Brauer algebra over 
 $\Lambda$, a certain localization of the ring of Laurent polynomials of two variables, 
is generated by the supersymmetric polynomials in some commuting elements, so called  \emph{Murphy operators}, which is a generalization of  the ones in \cite{DJ}.
It turns out that Morton's elements are natural deformation of our Jucys-Murphy elements (See Definition \ref{def:quantum JM} and Remark \ref{rem:quantum JM element}(2)). By taking a suitable limit sending $q$ to $1$, we can use our main theorem to show that the supersymmetric polynomials in those elements  generate the center of the quantized walled Brauer algebra over $\C(q,\rho)$ and hence Morton's conjecture holds provided the base ring is extended from $\Lambda$ to $\C(q,\rho)$.

This paper is organized as follows:
In Section 1, we briefly recall the definition of walled Brauer algebras $B_{r,s}(\delta)$ and
their cell modules.
In Section 2, we introduce the Jucys-Murphy elements $L_1, \ldots, L_{r+s}$ of $B_{r,s}(\delta)$
and  prove  several  relations between generators of $B_{r,s}(\delta)$ and the Jucys-Murphy elements.
Using these relations, we  show that the supersymmetric polynomials in $L_1, \ldots, L_{r+s}$ belong to the center of $B_{r,s}(\delta)$.
In Section 3, we calculate the eigenvalues of supersymmetric polynomials in the Jucys-Murphy elements on  cell modules.  Based on this calculation, we prove that when $B_{r,s}(\delta)$ is semisimple, the supersymmetric polynomials in $L_1, \ldots, L_{r+s}$ generates the center of $B_{r,s}(\delta)$.
 In Section 4, we define the Gelfand-Zetlin subalgebra of $B_{r,s}(\delta)$
 and show that it is generated by $L_1, \ldots, L_{r+s}$ when $B_{r,s}(\delta)$ is semisimple.
 In Section 5, we give an alternative proof of a part of the classification theorem of blocks of $B_{r,s}(\delta)$
 given in \cite{CDDM}.
 Lastly,  we present an analogue of Jucys-Murpy elements for the quantized walled Brauer algebra
$H_{r,s}(q,\rho)$ and we show that the supersymmetric polynomials in these elements generate the center. 

\vskip1em
\emph{Acknowledgements.} The authors would like to thank  Jonathan Brundan and Catharina Stroppel  for valuable discussions on their Jucys-Murphy elements, and thank 
  Antonio Sartori for 
explaining the modification in Remark \ref{rem:different}.
 They also thank  Hugh Morton for a valuable comment.
\vskip1em

\vskip1em
\section{Walled Brauer algebras and their Representations}

In this section, we will recall the  walled Brauer algebras and their  cell modules. We mainly follow the exposition in the papers \cite{BS,CDDM}.

\subsection{Walled Brauer algebras}
Let $r$ and $s$ be nonnegative integers.
  An \emph{$(r,s)$-walled Brauer diagram} is a graph consisting of two rows with $r+s$ vertices in each row such that
  the following conditions hold:
  \begin{enumerate}
\item[\rm{(1)}]  Each vertex is connected by a strand to exactly one other vertex.
\item[\rm{(2)}] There is a vertical wall which separates the first $r$ vertices from the last $s$ vertices in each row.
\item[{\rm (3)}]  A \emph{vertical strand} connects a vertex on the top row with one on the bottom row, and it
cannot cross the wall.  A  \emph{horizontal strand} connects vertices in the same row, and  it must cross the wall.
\end{enumerate}

Note that the vertical strands  are called the \emph{propagating lines} and the horizontal strands on the top row (respectively, on the bottom row) are called the \emph{northern arcs} (respectively, \emph{southern arcs})  in \cite{CDDM}.

Let $\delta$ be a complex number
and let us denote $B_{r,s}(\delta)$  the $\C$-vector space spanned by the basis consisting of all the $(r,s)$-walled Brauer diagrams.
The dimension of  $B_{r,s}(\delta)$ equals to  $(r+s) !$ (see, for example \cite[(2.2)]{BS}).
We define a multiplication of $(r,s)$-walled Brauer diagrams as follows:
For $(r,s)$-walled Brauer diagrams $d_1,d_2$,
we put $d_1$ under $d_2$ and identify the top vertices
of $d_1$ with the bottom vertices of $d_2$.
We remove the loops in the middle row, if there exist.
Then thus obtained diagram, denoted by $d_1 * d_2$, is again an $(r,s)$-walled Brauer diagram.
We define the multiplication of $d_1$ by $d_2$
\eqn
d_1 d_2 : = \delta^n \, d_1 * d_2 \in B_{r,s}(\delta),
\eneqn
where $n$ denotes the number of loops we removed in the middle row.
Extending this multiplication by linearity,
we obtain a multiplication on  $B_{r,s}(\delta)$, which can be easily seen to be associative.
We call thus obtained $\C$-algebra the \emph{walled Brauer algebra}.

For each $(r,s)$-walled Brauer diagram, we number the vertices in each row of it by $1,2,$ $\ldots,$ $r+s$ in the order from left to right.
Then we have the following set of generators $$\set{s_i}{1 \le i \le r-1} \cup\set{ s_j}{  r+1  \le j \le r+s-1} \cup \{e_{r,r+1}\}$$
of the algebra $B_{r,s}(\delta)$ given by
\vskip 1em
 ${\beginpicture
\setcoordinatesystem units <0.78cm,0.39cm>
\setplotarea x from 0 to 9, y from -1 to 4
\put{$s_i:= $} at 0 1.5
\put{$\bullet$} at  1 0  \put{$\bullet$} at  1 3
\put{$\bullet$} at  2 0  \put{$\bullet$} at  2 3
\put{$\bullet$} at  3 0  \put{$\bullet$} at  3 3
\put{$\bullet$} at  4 0  \put{$\bullet$} at  4 3
\put{$\bullet$} at  5 0  \put{$\bullet$} at  5 3
\put{$\bullet$} at  6 0  \put{$\bullet$} at  6 3
\put{$\bullet$} at  7 0  \put{$\bullet$} at  7 3
\put{$\bullet$} at  8 0  \put{$\bullet$} at  8 3

\put{$\cdots$} at 1.5 1.5
\put{$\cdots$} at 5.5 1.5
\put{$\cdots$} at 7.5 1.5
\put{{\scriptsize$i$}} at 3 4
\put{{\scriptsize$i+1$}} at 4 4
\put{,} at 8.5 1
\plot 1 3 1 0 /
\plot 2 3 2 0 /
\plot 3 3 4 0 /
\plot 4 3 3 0 /
\plot 5 3 5 0 /
\plot 6 3 6 0 /
\plot 7 3 7 0 /
\plot 8 3 8 0 /
\setdashes  <.4mm,1mm>
\plot 6.5 -1   6.5 4 /
\setsolid
\endpicture}$
${\beginpicture \setcoordinatesystem units <0.78cm,0.39cm>
\setplotarea x from 0 to 9, y from -1 to 4
\put{$s_j:= $} at 0 1.5
\put{$\bullet$} at  1 0
\put{$\bullet$} at  1 3
\put{$\bullet$} at 2 0
\put{$\bullet$} at  2 3
\put{$\bullet$} at  3 0
\put{$\bullet$} at 3 3
\put{$\bullet$} at  4 0  \put{$\bullet$} at 4 3
\put{$\bullet$} at  5 0  \put{$\bullet$} at  5 3
\put{$\bullet$} at 6 0 \put{$\bullet$} at  6 3
\put{$\bullet$} at  7 0 \put{$\bullet$} at 7 3
\put{$\bullet$} at  8 0  \put{$\bullet$} at 8 3

\put{$\cdots$} at 1.5 1.5
\put{$\cdots$} at 3.5 1.5
\put{$\cdots$} at 7.5 1.5
\put{,} at 8.5 1
\put{{\scriptsize$j$}} at 5 4
\put{{\scriptsize$j+1$}} at 6 4

\plot 1 3 1 0 /
\plot 2 3 2 0 /
\plot 3 3 3 0 /
\plot 4 3 4 0 /
\plot 5 3 6 0 /
\plot 6 3 5 0 /
\plot 7 3 7 0 /
\plot 8 3 8 0 /
\setdashes  <.4mm,1mm>
\plot 2.5 -1   2.5 4 /
\setsolid
\endpicture}$\\

$${\beginpicture
\setcoordinatesystem units <0.78cm,0.39cm>
\put{$e_{r,r+1} := $} at -0.5 1.5
\put{$\bullet$} at  1 0  \put{$\bullet$} at  1 3
\put{$\bullet$} at  3 0  \put{$\bullet$} at  3 3
\put{$\bullet$} at  4 0  \put{$\bullet$} at  4 3
\put{$\bullet$} at  5 0  \put{$\bullet$} at  5 3
\put{$\bullet$} at  6 0  \put{$\bullet$} at  6 3
\put{$\bullet$} at  8 0  \put{$\bullet$} at  8 3

\put{$\cdots$} at 2 1.5
\put{$\cdots$} at 7 1.5
\put{.} at 8.7 1

 \put{{\scriptsize $1$}} at 1 4
\put{{\scriptsize $r$}} at 4 4
\put{{\scriptsize $r+1$}} at 5 4
\put{{\scriptsize $r+s$}} at 8 4

\plot 1 3 1 0 /
\plot 3 3 3 0 /
\plot 8 3 8 0 /
\plot 6 3 6 0 /
\setdashes  <.4mm,1mm>
\plot 4.5 -1   4.5 4 /
\setsolid
\setquadratic
\plot 4 3 4.5 2 5 3 /
\plot 4 0 4.5 1 5 0  /
\endpicture}$$

\vskip1mm
Let $0 \le t, t' \le \min(r,s)$. The subalgebra of $B_{r,s}(\delta)$ generated by $s_1,\ldots, s_{r-t-1}$ and $s_{r+t'+1}, \ldots, s_{r+s-1}$ can be identified with the group algebra $\C[\sym_{r-t}  \times  \sym_{s-t'}] \simeq \C[\sym_{r-t}] \otimes \C[\sym_{s-t'}]$, where $\sym_k$ denotes  the symmetric group of $k$ letters.
We will use this identification for the rest of the paper.
Let us define
\eqn
&& (a,b)=(b,a):= s_{b-1} \cdots s_{a+1} s_a s_{a+1}  \cdots s_{b-1},
 \ \text{for } 1 \le a < b \le r \ \text{or} \  r+1 \le a < b \le r+s, \\
&&  e_{j,k} :=   (s_{k-1} \cdots s_{r+2} s_{r+1}) (s_j \cdots s_{r-2} s_{r-1}) e_{r,r+1}  (s_{r-1} s_{r-2} \cdots s_{j})(s_{r+1} s_{r+2} \cdots s_{k-1}) \\
&&   \quad \hskip23em \text{for} \ 1\le j \le r, \  r+1 \le k \le r+s.
\eneqn
Indeed, the above elements have simple forms as diagrams:
\vskip0.5em
$(a,b)=(b,a)= $
$$\hskip-2em{\beginpicture
\setcoordinatesystem units <0.78cm,0.39cm>
\setplotarea x from 0 to 9, y from -1 to 4
\put{$\bullet$} at  1 0  \put{$\bullet$} at  1 3
\put{$\bullet$} at  3 0  \put{$\bullet$} at  3 3
\put{$\bullet$} at  4 0  \put{$\bullet$} at  4 3
\put{$\bullet$} at  5 0  \put{$\bullet$} at  5 3
\put{$\bullet$} at  7 0  \put{$\bullet$} at  7 3
\put{$\bullet$} at  8 0  \put{$\bullet$} at  8 3
\put{$\bullet$} at  9 0  \put{$\bullet$} at  9 3
\put{$\bullet$} at  11 0  \put{$\bullet$} at  11 3
\put{$\bullet$} at  12 0  \put{$\bullet$} at  12 3
\put{$\bullet$} at  14 0  \put{$\bullet$} at  14 3

\put{{\scriptsize$a$}} at 4 4
\put{{\scriptsize$b$}} at 8 4
\put{{\scriptsize$1$}} at 1 4
\put{{\scriptsize$r$}} at 11 4
\put{{\scriptsize$r+1$}} at 12 4
\put{{\scriptsize$r+s$}} at 14 4

\put{$\cdots$} at 2 1.5
\put{$\cdots$} at 6 2
\put{$\cdots$} at 10 1.5
\put{$\cdots$} at 13 1.5

\put{if $1 \le a<b\le r,$} at 16.5 1.5

\plot 1 0 1 3 /
\plot 3 0 3 3 /
\plot 5 0 5 3 /
\plot 7 0 7 3 /
\plot 9 0 9 3 /
\plot 11 0 11 3 /
\plot 12 0 12 3 /
\plot 14 0 14 3 /

\plot 4 0 8 3 /
\plot 8 0 4 3 /

\setdashes  <.4mm,1mm>
\plot 11.5 -1   11.5 4 /
\setsolid
\endpicture}$$

\vskip0.5em
$$\ \ {\beginpicture
\setcoordinatesystem units <0.78cm,0.39cm>
\setplotarea x from 0 to 9, y from -1 to 4
\put{$\bullet$} at  1 0  \put{$\bullet$} at  1 3
\put{$\bullet$} at  3 0  \put{$\bullet$} at  3 3
\put{$\bullet$} at  4 0  \put{$\bullet$} at  4 3
\put{$\bullet$} at  6 0  \put{$\bullet$} at  6 3
\put{$\bullet$} at  7 0  \put{$\bullet$} at  7 3
\put{$\bullet$} at  8 0  \put{$\bullet$} at  8 3
\put{$\bullet$} at  10 0  \put{$\bullet$} at  10 3
\put{$\bullet$} at  11 0  \put{$\bullet$} at  11 3
\put{$\bullet$} at  12 0  \put{$\bullet$} at  12 3
\put{$\bullet$} at  14 0  \put{$\bullet$} at  14 3

\put{{\scriptsize$a$}} at 7 4
\put{{\scriptsize$b$}} at 11 4

\put{{\scriptsize$1$}} at 1 4
\put{{\scriptsize$r$}} at 3 4
\put{{\scriptsize$r+1$}} at 4 4
\put{{\scriptsize$r+s$}} at 14 4

\put{$\cdots$} at 2 1.5
\put{$\cdots$} at 5 1.5
\put{$\cdots$} at 9 2
\put{$\cdots$} at 13 1.5

\put{if $r+1 \le a<b\le r+s,$} at 17 1.5

\plot 1 0 1 3 /
\plot 3 0 3 3 /
\plot 4 0 4 3 /
\plot 6 0 6 3 /
\plot 8 0 8 3 /
\plot 10 0 10 3 /
\plot 12 0 12 3 /
\plot 14 0 14 3 /

\plot 7 0 11 3 /
\plot 11 0 7 3 /

\setdashes  <.4mm,1mm>
\plot 3.5 -1   3.5 4 /
\setsolid
\endpicture}$$
and
\vskip0.5em
$ \  {\beginpicture
\setcoordinatesystem units <0.78cm,0.39cm>
\setplotarea x from 0 to 9, y from -1 to 4
\put{$e_{j,k}= $} at 0 1.5
\put{$\bullet$} at  1 0  \put{$\bullet$} at  1 3
\put{$\bullet$} at  3 0  \put{$\bullet$} at  3 3
\put{$\bullet$} at  4 0  \put{$\bullet$} at  4 3
\put{$\bullet$} at  5 0  \put{$\bullet$} at  5 3
\put{$\bullet$} at  7 0  \put{$\bullet$} at  7 3
\put{$\bullet$} at  8 0  \put{$\bullet$} at  8 3
\put{$\bullet$} at  10 0  \put{$\bullet$} at  10 3
\put{$\bullet$} at  11 0  \put{$\bullet$} at  11 3
\put{$\bullet$} at  12 0  \put{$\bullet$} at  12 3
\put{$\bullet$} at  14 0  \put{$\bullet$} at  14 3

\put{{\scriptsize$j$}} at 4 4
\put{{\scriptsize$k$}} at 11 4

\put{{\scriptsize$1$}} at 1 4
\put{{\scriptsize$r$}} at 7 4
\put{{\scriptsize$r+1$}} at 8 4
\put{{\scriptsize$r+s$}} at 14 4

\put{$\cdots$} at 2 1.5
\put{$\cdots$} at 6 1.5
\put{$\cdots$} at 9 1.5
\put{$\cdots$} at 13 1.5

\put{.} at 14.7 1

\plot 1 0 1 3 /
\plot 3 0 3 3 /
\plot 5 0 5 3 /
\plot 7 0 7 3 /
\plot 8 0 8 3 /
\plot 10 0 10 3 /
\plot 12 0 12 3 /
\plot 14 0 14 3 /

\setdashes  <.4mm,1mm>
\plot 7.5 -1   7.5 4 /
\setsolid

\setquadratic
\plot 4 0 7.5 1 11 0 /
\plot 4 3 7.5 2 11 3 /
\endpicture}$

\noindent Note that in \cite{BS}, the elements $(a,b)$ and $-e_{j,k}$ are called the \emph{transpositions in $B_{r,s}(\delta)$}.

There is a criterion of the semisimplicity for $B_{r,s}(\delta)$:
\begin{prop} {\rm(}\cite[Theorem 6.3]{CDDM}{\rm )} \label{prop:semisimplicity}
The walled Brauer algebra $B_{r,s}(\delta)$ is semisimple if and only if
one of the followings holds:
\begin{enumerate}[\rm(1)]
\item $r=0$ or $s=0$,
\item $\delta \notin \Z$,
\item $|\delta| > r+s-2$,
\item $\delta=0, \quad \text{and} \quad (r,s) \in \{(1,2),  (1,3),  (2,1),  (3,1)\}$.
\end{enumerate}
\end{prop}
 Hence the algebra $B_{r,s}(\delta)$ is semisimple  except for finitely many values $\delta \in \C$ for a fixed pair $(r,s)$.
Note that an analogue of the above for the quantized walled Brauer algebra is proved in \cite[Theorem 6.10]{RS}.

\subsection{Cell modules}
A \emph{partition} is a weakly decreasing sequence of nonnegative integers with finitely many nonzero entries.
For a partition $\mu=(\mu_1,\mu_2,\cdots)$, set $|\mu|:=\sum_{i \ge 1} \mu_i$ and set $\ell(\mu):=|\set{i \ge 1}{ \mu_i \neq 0}|$.
Sometimes, we identify a partition $\mu=(\mu_1,\mu_2,\cdots)$ with the Young diagram
whose boxes are arranged in left-justified rows with the lengths $\mu_1,\mu_2,\cdots$, and denote it by
$[\mu]$.

Let us denote $\Lambda$ the set of pair of partitions.
An element in $\Lambda$ is called a \emph{bipartition}.
Set
\eqn
&&\Lambda^t_{r,s}:=\set{(\la^L,\la^R) \in \Lambda}{|\la^L|=r-t, \ |\la^R|=s-t}, \quad \Lambda_{r,s}:=\bigsqcup_{t=0}^{\min(r,s)} \Lambda_{r,s}^t, \\
&&\dot{\Lambda}_{r,s}
=\begin{cases}
\Lambda_{r,s} & \text{if} \ \delta \neq 0 \ \text{or} \ r \neq s \ \text{or} \ r=s=0, \\
\Lambda_{r,s}-\{(\emptyset, \emptyset)\} & \text{if} \ \delta=0 \ \text{and} \ r=s \neq 0.
\end{cases}
\eneqn

Define $J^k_{r,s}$ to be the subspace of $B_{r,s}(\delta)$ spanned by the $(r,s)$-walled Brauer diagrams with
at least $k$ horizontal strands at the top and at the bottom. Then $J^k_{r,s}$ is a two-sided ideal of $B_{r,s}(\delta)$ and $J^{k'}_{r,s} \subset J^{k}_{r,s}$ for $k' \ge k$.

Let $\lambda=(\la^L,\la^R)$ be an element in $\La_{r,s}^t$ for some $0   \le t \le \min(r,s)$.
Consider the subspace $I^t_{r,s}$ of the quotient space $J^t_{r,s}/J^{t+1}_{r,s}$ spanned by $Y_{r,s}^t$, where $Y_{r,s}^t$ denote  the set of images of
the diagrams
with exactly $t$ horizontal strands at the top connecting the $(r+1-k)$-th vertex to the $(r+k)$-th vertex for each $k=1,2\ldots,t$.
 We have $\dim I_{r,s}^t=  {r \choose  t } {s \choose t} t ! (r-t)! (s-t)!$.
Then the space $I^t_{r,s}$ has a $(B_{r,s}(\delta),\C[\sym_{r-t}] \otimes \C[\sym_{s-t}] )$-bimodule structure given by the left multiplication and the right multiplication, respectively.
Now we define
\eqn
C_{r,s}(\la):=I_{r,s}^t \otimes_{\C[\sym_{r-t}] \otimes \C[\sym_{s-t}]} S(\la^L) \boxtimes S(\la^R),
\eneqn
where $S(\la^L)$  denotes the simple $\sym_{r-t}$-module parametrized the partition $\la^L$ of $r-t$,
and  $S(\la^R)$   denotes the simple $\sym_{s-t}$-module parametrized the partition $\la^R$ of $s-t$  (see, for example, \cite{JK}).
Note that the bimodule $I_{r,s}^t$ is free over $\C[\sym_{r-t}] \otimes \C[\sym_{s-t}]$ with basis
$X^t_{r,s}$ consisting of elements in $Y_{r,s}^t$ in which no two vertical strands cross.
 The cardinality of  $X_{r,s}^t$ is ${r \choose t} {s \choose t} t !$.
Thus as vector spaces,
\eqn
C_{r,s}(\lambda) = \soplus_{\tau \in X_{r,s}^t} \tau \otimes (S(\la^L) \boxtimes S(\la^R)).
\eneqn

For $\la \in \dot{\Lambda}_{r,s}$, the cell module $C_{r,s}(\la)$ has an irreducible head $D_{r,s}(\la)$ and the family
$\set{D_{r,s}(\la)}{\la \in \dot{\Lambda}_{r,s}}$
is the complete set of mutually  non-isomorphic  simple modules over $B_{r,s}(\delta)$ (see \cite[Theorem 2.7]{CDDM}).
The elements in ${\Lambda}_{r,s}$ are called the \emph{weights}.

The cell modules are indecomposable. Moreover, we have
\begin{lemma} {\rm(}\cite[Lemma 2.2]{BS}\rm{)} \label{lem:one dimensional}
For any $\la \in \Lambda_{r,s}$ we have
\eqn
\End_{B_{r,s}(\delta)}(C_{r,s}(\la)) \simeq \C.
\eneqn
\end{lemma}

\vskip1em
\section{Jucys-Murphy elements and supersymmetric polynomials}
\subsection{Jucys-Murphy elements}
\begin{definition} \label{def:JM element}
 {\rm For each $1 \le k \le r+s$, we define
\eqn
L_k : = \begin{cases}
\displaystyle \sum_{ j=1}^{k-1} (j,k) & \text{if } \ 1 \le k \le r, \\
\displaystyle -\sum_{j=1}^{r} e_{j,k} +\sum_{j= r+1}^{k-1}(j, k) +\delta  & \text{if } \ r+1 \le k \le r+s.
\end{cases}
\eneqn
We call these $L_k$'s the \emph{Jucys-Murphy elements of $B_{r,s}(\delta)$}.  In particular, $L_1=0$.}
\end{definition}

\newcommand{\idem}{\boldsymbol{1}}
\newcommand{\bolda}{{\boldsymbol{\mathrm{a}}}}
\newcommand{\boldb}{{\boldsymbol{\mathrm{b}}}}
\begin{remark} \label{rem:different}
The above Jucys-Murphy elements are similar to those in \cite{BS} and those in \cite{SartoriStroppel}, but different.
Precisely speaking, the parameter $\delta$ does not appear in the definitions of Jucys-Murphy elements in \cite{BS, SartoriStroppel}.
Actually, if we modify the definition of $\xi_1 \idem_{\bolda}$, which was defined as  $0 \idem_{\bolda}$  in \cite[(2.7)]{SartoriStroppel}, into
\eqn
\xi_1 \idem_{\bolda} =\begin{cases}
 0\idem_{\bolda}  & \text{if }\ a_1=\uparrow, \\
 \delta \idem_{\bolda} & \text{if }\ a_1=\downarrow,
\end{cases}
\eneqn
then  the elements $\xi_k\idem_{\uparrow^r \downarrow^s}$ in \cite[(2.7)]{SartoriStroppel}  yields  our  element $L_k$ of $B_{r,s}(\delta)$. This modification is due to Antonio Sartori.

Note  that there are so called \emph{Jucys-Murphy-like elements} of walled Brauer algebras, introduced in \cite{RuiSu}, which are still different from ours.
\end{remark}

The following relations will be used frequently.
\begin{lemma} \label{lem:relations between horizontal lines and vertical lines}
For all mutually distinct and admissible $i, j, i',j'$,  we have
\begin{enumerate}[\rm(1)]
\item $(i',j')e_{i,j} =e_{i,j} (i',j')$,
\item $(i,i')e_{i,j} =e_{i',j} (i,i')$ and  \ $(j,j') e_{i,j} = e_{i,j'} (j,j')$,
\item $e_{i,j} e_{i',j'} = e_{i',j'} e_{i,j} $,
\item $e_{i,j} e_{i,j'} = e_{i,j} (j,j')$ and \ $e_{i,j} e_{i'j} = e_{i,j} (i,i')$,
\item $e_{i,j} e_{i',j'}(i,i')=e_{i,j} e_{i',j'}(j,j')$,
\end{enumerate}
\end{lemma}
\Proof
 They can be checked by direct calculations on $(r,s)$-walled Brauer diagrams.
For example, we can check the first equality in (4) as follows:
\vskip1em
$${\beginpicture
\setcoordinatesystem units <0.78cm,0.39cm>
\setplotarea x from 0 to 9, y from -1 to 4
\put{$e_{i,j}e_{i,j'}= $} at -0.5 3
\put{$\bullet$} at  1 0  \put{$\bullet$} at  1 3
\put{$\bullet$} at  2 0  \put{$\bullet$} at  2 3
\put{$\bullet$} at  3 0  \put{$\bullet$} at  3 3

\put{$\bullet$} at  1 3.5  \put{$\bullet$} at  1 6.5
\put{$\bullet$} at  2 3.5  \put{$\bullet$} at  2 6.5
\put{$\bullet$} at  3 3.5  \put{$\bullet$} at  3 6.5

\put{{\scriptsize$i$}} at 1 -1
\put{{\scriptsize$j$}} at 2 -1
\put{{\scriptsize$j'$}} at 3 -1

\plot 3 0 3 3 /
\plot 2 3.5 2 6.5 /

\setdashes  <.4mm,1mm>
\plot 1.5 -1   1.5 7.5 /
\setsolid

\put{$=$} at 4 3
\put{$\bullet$} at  5 1.5  \put{$\bullet$} at  5 4.5
\put{$\bullet$} at  6 1.5  \put{$\bullet$} at  6 4.5
\put{$\bullet$} at  7 1.5  \put{$\bullet$} at  7 4.5

\put{{\scriptsize$i$}} at 5 0.5
\put{{\scriptsize$j$}} at 6 0.5
\put{{\scriptsize$j'$}} at 7 0.5

\plot 7 1.5  6 4.5 /

\setdashes  <.4mm,1mm>
\plot 5.5 0.5   5.5 5.5 /
\setsolid
\put{$=$} at 8 3
\put{$\bullet$} at  9 0  \put{$\bullet$} at  9 3
\put{$\bullet$} at  10 0  \put{$\bullet$} at  10 3
\put{$\bullet$} at  11 0  \put{$\bullet$} at  11 3

\put{$\bullet$} at  9 3.5  \put{$\bullet$} at  9 6.5
\put{$\bullet$} at  10 3.5  \put{$\bullet$} at  10 6.5
\put{$\bullet$} at  11 3.5  \put{$\bullet$} at  11 6.5

\put{{\scriptsize$i$}} at 9 -1
\put{{\scriptsize$j$}} at 10 -1
\put{{\scriptsize$j'$}} at 11 -1

\plot 11 0 11 3 /
\plot 9 3.5 9 6.5 /
\plot 10 3.5 11 6.5 /
\plot 11 3.5 10 6.5 /

\setdashes  <.4mm,1mm>
\plot 9.5 -1   9.5 7.5 /
\setsolid

\put{$= e_{i,j}(j,j')$ if $j <j'$,} at 14 3

\setquadratic
\plot 1 3 1.5 2 2 3 /
\plot 1 0 1.5 1 2 0  /
\plot 1 3.5 2 4.5 3 3.5 /
\plot 1 6.5 2 5.5 3 6.5  /

\setquadratic
\plot 5 1.5 5.5 2.5 6 1.5 /
\plot 5 4.5 6 3.5 7 4.5  /

\setquadratic
\plot 9 0 9.5 1 10 0 /
\plot 9 3 9.5 2 10 3 /
\endpicture}$$

\vskip1.5em
$${\beginpicture
\setcoordinatesystem units <0.78cm,0.39cm>
\setplotarea x from 0 to 9, y from -1 to 4
\put{$e_{i,j}e_{i,j'}= $} at -0.5 3
\put{$\bullet$} at  1 0  \put{$\bullet$} at  1 3
\put{$\bullet$} at  2 0  \put{$\bullet$} at  2 3
\put{$\bullet$} at  3 0  \put{$\bullet$} at  3 3

\put{$\bullet$} at  1 3.5  \put{$\bullet$} at  1 6.5
\put{$\bullet$} at  2 3.5  \put{$\bullet$} at  2 6.5
\put{$\bullet$} at  3 3.5  \put{$\bullet$} at  3 6.5

\put{{\scriptsize$i$}} at 1 -1
\put{{\scriptsize$j'$}} at 2 -1
\put{{\scriptsize$j$}} at 3 -1

\plot 2 0 2 3 /
\plot 3 3.5 3 6.5 /

\setdashes  <.4mm,1mm>
\plot 1.5 -1   1.5 7.5 /
\setsolid

\put{$=$} at 4 3
\put{$\bullet$} at  5 1.5  \put{$\bullet$} at  5 4.5
\put{$\bullet$} at  6 1.5  \put{$\bullet$} at  6 4.5
\put{$\bullet$} at  7 1.5  \put{$\bullet$} at  7 4.5

\put{{\scriptsize$i$}} at 5 0.5
\put{{\scriptsize$j'$}} at 6 0.5
\put{{\scriptsize$j$}} at 7 0.5

\plot 6 1.5  7 4.5 /

\setdashes  <.4mm,1mm>
\plot 5.5 0.5   5.5 5.5 /
\setsolid
\put{$=$} at 8 3
\put{$\bullet$} at  9 0  \put{$\bullet$} at  9 3
\put{$\bullet$} at  10 0  \put{$\bullet$} at  10 3
\put{$\bullet$} at  11 0  \put{$\bullet$} at  11 3

\put{$\bullet$} at  9 3.5  \put{$\bullet$} at  9 6.5
\put{$\bullet$} at  10 3.5  \put{$\bullet$} at  10 6.5
\put{$\bullet$} at  11 3.5  \put{$\bullet$} at  11 6.5

\put{{\scriptsize$i$}} at 9 -1
\put{{\scriptsize$j'$}} at 10 -1
\put{{\scriptsize$j$}} at 11 -1

\plot 10 0 10 3 /
\plot 9 3.5 9 6.5 /
\plot 10 3.5 11 6.5 /
\plot 11 3.5 10 6.5 /

\setdashes  <.4mm,1mm>
\plot 9.5 -1   9.5 7.5 /
\setsolid

\put{$= e_{i,j}(j,j')$ if $j' < j$.} at 14 3

\setquadratic
\plot 1 0 2 1 3 0  /
\plot 1 3 2 2 3 3 /
\plot 1 3.5 1.5 4.5 2 3.5 /
\plot 1 6.5 1.5 5.5 2 6.5  /

\setquadratic
\plot 5 1.5 6 2.5 7 1.5 /
\plot 5 4.5 5.5 3.5 6 4.5  /

\setquadratic
\plot 9 0 10 1 11 0 /
\plot 9 3 10 2 11 3 /
\endpicture}$$
Since the other relations can be checked similarly, we omit  their  proofs.
\QED

The following proposition will play an important role in  the rest of this paper.
Note that  similar  relations in Brauer algebras appeared in \cite[Proposition 2.3]{Nazarov}.
\begin{prop}  \label{prop:commuting relations with generators}
We have the following relations:
\begin{enumerate}[\rm(1)]
\item  $(i, i+1) L_{i+1} -L_i(i,i+1)=1$ for $i \neq r$,
 \item $L_{i+1} (i,i+1)-(i,i+1) L_i=1$ for $i \neq r$,
  \item $e_{r, r+1} (L_r +L_{r+1}) =(L_r +L_{r+1}) e_{r, r+1} =0,$
  \item $(i, i+1) L_k = L_k (i,i+1)$ for $k \neq i,i+1$,
  \item $e_{r, r+1} L_k =L_k e_{r,r+1}$ for $k \neq r, r+1$.
  \end{enumerate}
\end{prop}

\Proof
(1) When $1 \le i \le r-1$, the relation (1) is well-known. Assume that $r+1 \le i \le r+s-1.$
Then we have
\eqn
(i,i+1)L_{i+1} && =(i,i+1) \Big( - \sum_{j=1}^{r} e_{j,i+1} +\sum_{j=r+1}^{i} (j,i+1) +\delta \Big) \\
               && = - \sum_{j=1}^{r} e_{j,i} (i,i+1) +\sum_{j=r+1}^{i-1} (i,i+1)(j,i+1) +1 +\delta (i,i+1) \\
               && = \Big(- \sum_{j=1}^{r} e_{j,i} +\sum_{j=r+1}^{i-1} (j,i) +\delta \Big)(i,i+1) +1 =L_i (i, i+1) +1.
\eneqn

(2) It follows from (1) that
\eqn
L_{i+1}(i,i+1) = (i,i+1) (L_i (i, i+1)+1)(i,i+1)=(i,i+1)L_i+1.
\eneqn

(3) By direct calculations and Lemma \ref{lem:relations between horizontal lines and vertical lines} (4), we  get
\eqn
e_{r,r+1}(L_r +L_{r+1}) && = e_{r,r+1} \Big( \sum_{j=1}^{r-1} (j,r) -\sum_{j=1}^r e_{j,r+1} + \delta \Big) \\
                        && = \sum_{j=1}^{r-1} e_{r,r+1}(j,r) - \sum_{j=1}^{r-1} e_{r,r+1}(j,r) - \delta e_{r,r+1} + \delta e_{r,r+1} =0.
\eneqn
Similarly, by Lemma \ref{lem:relations between horizontal lines and vertical lines} (2) and (4), we have
\eqn
(L_r +L_{r+1})e_{r,r+1} && = \Big( \sum_{j=1}^{r-1} (j,r) -\sum_{j=1}^r e_{j,r+1} + \delta \Big) e_{r,r+1} \\
                        && = \sum_{j=1}^{r-1} e_{j,r+1}(j,r) - \sum_{j=1}^{r-1} e_{j,r+1}(j,r) - \delta e_{r,r+1} + \delta e_{r,r+1} =0.
\eneqn

(4) For $i \ge k+1$, it is trivial that $L_k (i,i+1) =(i,i+1) L_k$.
For  $ i<k-1<k \le  r $, it is well-known that $(i,i+1)$ commutes with $L_k$.
Assume that $i< r< k$. Then  from Lemma \ref{lem:relations between horizontal lines and vertical lines} (1),(2) and direct calculations, we have
\eqn
(i,i+1) L_k && = (i,i+1) \Big( -\sum_{j=1}^r e_{j,k} + \sum_{j=r+1}^{k-1} (j,k)+ \delta \Big) \\
            && =  -\sum_{\substack{j=1\\j \neq i,i+1}}^r e_{j,k}(i,i+1)  -e_{i+1,k}(i,i+1) -e_{i,k}(i,i+1)\\
            && \quad +\sum_{j=r+1}^{k-1} (j,k)(i,i+1) +\delta (i,i+1) \\
            && = \Big(-\sum_{j=1}^r e_{j,k} +\sum_{j=r+1}^{k-1} (j,k) + \delta \Big) (i,i+1) =L_k (i,i+1).
\eneqn

Assume that $ r< i < k-1$.
Then we can check that
\eqn
(i,i+1)L_k && = (i,i+1) \Big( -\sum_{j=1}^r e_{j,k} + \sum_{j=r+1}^{k-1} (j,k)+ \delta \Big) \\
           && = -\sum_{j=1}^r e_{j,k} (i,i+1) +\sum_{\substack{j=r+1\\ j \neq i, i+1}}^{k-1} (j,k)(i,i+1) +\delta(i,i+1) \\
           && \quad + (i,i+1)(i,k) +(i,i+1)(i+1, k)  \\
           &&= -\sum_{j=1}^r e_{j,k} (i,i+1) +\sum_{\substack{j=r+1\\ j \neq i, i+1}}^{k-1} (j,k)(i,i+1) +\delta(i,i+1) \\
           && \quad + (i+1,k)(i,i+1) +(i, k)(i,i+1) \\
           && = \Big( -\sum_{j=1}^r e_{j,k} + \sum_{j=r+1}^{k-1} (j,k)+ \delta \Big) (i,i+1) =L_k (i,i+1).
\eneqn

(5) If $k < r$, it is trivial that $e_{r,r+1} L_k =L_k e_{r,r+1}$. Let $k \ge r+2$.
Then we have
\eqn
e_{r,r+1}L_k && =e_{r,r+1} \Big(-\sum_{j=1}^r e_{j,k} + \sum_{j=r+1}^{k-1} (j,k)+ \delta \Big) \\
             && = -\sum_{j=1}^{r-1} e_{r,r+1}e_{j,k} -e_{r,r+1}(r+1, k) +\sum_{j=r+1}^{k-1} e_{r,r+1}(j,k) +\delta e_{r,r+1} \\
             && = -\sum_{j=1}^{r-1} e_{j,k}e_{r,r+1}  +\sum_{j=r+2}^{k-1} e_{r,r+1}(j,k) +\delta e_{r,r+1}\\
             && =\Big(-\sum_{j=1}^r e_{j,k} + \sum_{j=r+1}^{k-1} (j,k)+ \delta \Big) e_{r,r+1} =L_k e_{r,r+1},
\eneqn
as desired. In the fourth equality, we use  $e_{r,k} e_{r,r+1} = e_{r,k}(r+1, k)=(r+1,k)e_{r,r+1}$ obtained by Lemma \ref{lem:relations between horizontal lines and vertical lines} (2) and (4).
\QED

\begin{remark} \label{rem:variants}
For each $a \in \C$,  we have variants of the Jucys-Murphy elements:
set
\eqn
L^a_k:=\begin{cases} L_k + a & \text{if } \ 1 \le k \le r, \\
L_k-a
 & \text{if } \ r+1 \le k \le r+s.
\end{cases}
\eneqn
It is easy to see that  the above proposition holds with $L^a_k$ instead of $L_k$.
All the other parts of the rest of this paper will be  also valid, after a small modification.
\end{remark}

\begin{prop}    \label{prop:Li are commuting}
The elements $L_k$'s are commuting to each other.
\end{prop}

\Proof
It is well-known that the elements $L_1, \ldots, L_r$ are commuting to each other.
Let $B_r$  be the subalgebra of $B_{r,s}(\delta)$ generated by
$s_1, \ldots, s_{r-1}$ and
let $B_{r+a}$  be the subalgebra generated by $s_1, \ldots, s_{r-1}, e_{r,r+1}, s_{r+1}, \ldots s_{r+a-1}$
 for $1 \le a \le s$.
Then we have  $L_1,\ldots,L_{r+a-1}, L_{r+a} \in B_{r+a}$ for $0 \le a \le s$.
On the other hand, by Proposition \ref{prop:commuting relations with generators} (4) and (5), we know
that the element $L_{r+a}$ commutes with the generators of  $B_{r+a-1}$ for each $1 \le a \le s$.
Therefore, $L_{r+a}$  commutes with $L_{1}, \ldots, L_{r+a-1}$, as desired.
\QED

\begin{prop}  \label{prop:belongtocenter}
For each $k \in \Z_{\ge 0}$, the element
$$L_1^k+\cdots+L_r^k+(-1)^{k+1}(L_{r+1}^k+\cdots+L_{r+s}^k)$$
  belongs to the center of $B_{r,s}(\delta)$.
\end{prop}
\Proof
From Proposition \ref{prop:commuting relations with generators} (1) and (2), we have
\eqn
&&(i,i+1)(L_i +L_{i+1})=(L_i +L_{i+1})(i,i+1), \\
&&(i,i+1)(L_i L_{i+1})=(L_{i+1} L_i)(i,i+1)=(L_i L_{i+1})(i,i+1)
\eneqn
for $i \neq r$.
  Thus $(i,i+1)$ commutes with every symmetric polynomials in $L_i$ and $L_{i+1}$. In particular, it commutes with the power sum symmetric polynomials $L_i^k+L_{i+1}^k$ $(k \ge 0)$.
 Combining  this fact with Proposition \ref{prop:commuting relations with generators} (4),
it follows that
\eqn
&&(i,i+1) (L_1^k +\cdots + L_{r}^k) = (L_1^k + \cdots + L_r^k)(i,i+1) \\
&&(i,i+1)(L_{r+1}^k + \cdots +L_{r+s}^k) = (L_{r+1}^k + \cdots + L_{r+s}^k) (i, i+1)
\eneqn
 for all 
$i \neq r$.

From Proposition \ref{prop:commuting relations with generators} (3), we obtain
\eqn
e_{r,r+1}(L_r^k +(-1)^{k+1}L_{r+1}^k) =0 = (L_r^k +(-1)^{k+1} L_{r+1}^k)  e_{r,r+1}.
\eneqn
Hence, by using Proposition \ref{prop:commuting relations with generators} (5), we 
 conclude that
$L_1^k+\cdots+L_r^k+(-1)^{k+1}(L_{r+1}^k+\cdots+L_{r+s}^k)$ is in the center of $B_{r,s}(\delta)$,
 since it commutes with all the generators.
\QED
When $k=1$,  the above was shown in \cite[Lemma 2.1]{BS}. Indeed, the element $z_{r,s}$ in \cite[Lemma 2.1]{BS} is the same as $L_1+ \cdots +L_{r+s} -s \delta$.

\subsection{Supersymmetric polynomials}
In this section we recall the notion of supersymmetric polynomials. For details on supersymmetric polynomials, see for example, \cite{Moens}.

\begin{definition} \label{def:supersymmetric polynomial}
{\rm Let $r,s$ be nonnegative integers.
We say that an element $p$ in the polynomial ring  $\C[x_1,\ldots,x_r,y_1\ldots,y_s]$ is \emph{supersymmetric} if
\begin{enumerate}
\item[\rm {(1)}] $p$ is \emph{doubly symmetric}; i.e, it is symmetric in $x_1,\ldots,x_r$ and $y_1,\ldots, y_s$ separately,
\item[\rm{(2)}] $p$ satisfies the \emph{cancellation property}; i.e., the substitution $x_r=-y_1=t$ yields a polynomial in $x_1,\ldots,x_{r-1}, y_2,\ldots, y_s$ which is independent of $t$.
\end{enumerate}
We denote $S_{r,s}[x;y]$ the set of supersymmetric polynomials in $x_1,\ldots,x_r,y_1\ldots,y_s$.
It is a $\C$-subalgebra of  $\C[x_1,\ldots,x_r,y_1\ldots,y_s]^{\sym_r\times \sym_s}$, the algebra of doubly symmetric polynomials.}
\end{definition}

For $k \ge 0$, the \emph{ $k$-th \ power sum supersymmetric polynomial} is given by
\eqn
p_k(x_1,\ldots,x_r,y_1,\ldots,y_s):=x_1^k+\cdots +x_r^k+(-1)^{k+1}(y_1^k+\cdots +y_s^k).
\eneqn
It is easy to see that $p_k$ belongs to $S_{r,s}[x;y]$.
In \cite{Stem}, Stembridge showed that $S_{r,s}[x;y]$ is generated  by $\set{p_k}{k \ge 0}$.
Hence the following is an immediate consequence of Proposition \ref{prop:belongtocenter}.
\begin{corollary} \label{cor:centerssym}
For every supersymmetric polynomial $p$ in $S_{r,s}[x;y]$, the element
$$p(L_1,\ldots,L_{r+s})$$ belongs to the center $Z(B_{r,s}(\delta))$ of $B_{r,s}(\delta)$.
\end{corollary}

\begin{remark} If we take the
modification in Remark \ref{rem:different} and focus on the case of $B_{r,s}(\delta)$,
then the above corollary corresponds to \cite[Corollary 7.2]{SartoriStroppel},
Note that in \cite[Corollary 7.2]{SartoriStroppel} they used a description of the center of the degenerate affine walled Brauer algebra.
\end{remark}

The \emph{elementary supersymmetric polynomials} $$e_k(x_1,\ldots,x_r,y_1,\ldots,y_s) \quad (k \in  \Z_{\ge 0})$$
 are given by the generating function
\eqn
\sum_{k=0}^\infty  e_k(x_1,\ldots,x_r,y_1,\ldots,y_s) z^k
=\dfrac{\prod_{i=1}^r (1+x_i z)}{\prod_{j=1}^s (1-y_j z)}.
\eneqn
It is also known that $\set{e_k}{k \in \Z_{\ge 0}}$ generates the ring of supersymmetric polynomials (\cite[Corollary]{Stem}).
Then the lemma below follows immediately.

\begin{lemma} \label{lem:distinct}
Let $(a_1,\ldots,a_r,b_1,\ldots,b_s)$ and $(c_1,\ldots,c_r,d_1,\ldots,d_s)$ be
 elements in $\C^{r+s}$.
Then the followings are equivalent.
\begin{enumerate}[\rm(1)]
\item For every supersymmetric polynomial $p \in S_{r,s}[x;y]$, we have
\eqn
p(a_1,\ldots,a_r,b_1,\ldots,b_s)= p(c_1,\ldots,c_r,d_1,\ldots,d_s).
\eneqn
\item We have an equality
\eqn
\dfrac{\prod_{i=1}^r (1+a_i z)}{\prod_{j=1}^s (1-b_j z)}
=\dfrac{\prod_{i=1}^r (1+c_i z)}{\prod_{j=1}^s (1-d_j z)}
\eneqn
of rational functions in $z$.

\end{enumerate}
\end{lemma}

For a partition $\mu=(\mu_1,\mu_2,\ldots)$, a filling of $\mu$ with entries $1,\ldots, |\mu|$ is called a
\emph{standard tableau of shape $\mu$}, when the entries in each row and each column are strictly increasing, from left to right and from top to bottom, respectively.
Let $\mathfrak t^\mu$ be the standard tableau such that the entries $1,2,\ldots,|\mu|$ appear in increasing order from left to right along successive rows.
Recall that, for a box $u$ in $[\mu]$, \emph{the content of $u$} is given by the integer $b-a$, where $u$ is located $(a,b)$-position in $[\mu]$.
For $1 \le i \le |\mu|$, we define $\cont(\mu, i)$ to be
the  content of the box in $\mu$ with entry $i$ in the tableau $\mathfrak t^\mu$.
It is called the \emph{content of $\mu$ at $i$}.
Note that the multiset $\set{\cont(\mu, i)}{1\le i \le |\mu|}$ determines the Young diagram $[\mu]$ and
hence the partition $\mu$.

For each $\la \in \La_{r,s}^t$,
set
\eqn
c(\la,i) := \begin{cases}
\cont(\la^L,i) & \text{if }\ 1 \le i \le r-t,\\
0 & \text{if }\ r-t+1 \le i \le r+t, \\
\cont(\la^R,i-r-t)+\delta & \text{if }\ r+t+1 \le i \le r+s.
\end{cases}
\eneqn

\begin{lemma} \label{lem:separation}
 Assume that $B_{r,s}(\delta)$ is semisimple.
If $\la \neq \mu$ for $\la, \mu \in \La_{r,s}$, then there is a supersymmetric polynomial $p^{\la, \mu}$ such that
$$p^{\la, \mu}((c(\la,i))_{1\le i \le r+s}) \neq p^{\la, \mu}((c(\mu,i))_{1\le i \le r+s}).$$
\end{lemma}

\begin{remark}
If we combine \cite[Corollary 7.7]{CDDM} with Lemma \ref{lem:equiv_balanced} in the last section, then the above lemma follows.
But we include the following proof for completeness.
\end{remark}

\Proof
Let $\la \in \Lambda_{r,s}^t$, $\mu \in \Lambda_{r,s}^{t'}$ for some $0 \le  t,t'  \le \min(r,s)$.

Suppose that
$$p((c(\la,i))_{1\le i \le r+s}) =  p((c(\mu,i))_{1\le i \le r+s})$$
for every $p \in S_{r,s}[x;y]$.
Equivalently,
\eq \label{eq:contents rational function}
&&\dfrac{\prod_{i=1}^{r-t} (1+\cont(\la^L,i) z)}{\prod_{j=r+t+1}^{r+s} (1-(\cont(\la^R,j)+\delta) z)}
=\dfrac{\prod_{i=1}^{r-t'} (1+\cont(\mu^L,i) z)}{\prod_{j=r+t'+1}^{r+s} (1-(\cont(\mu^R,j)+\delta) z)}.
\eneq

We shall show that
$\lambda=\mu$ for all cases in Proposition \ref{prop:semisimplicity}.

{\bf Case (1) :} Assume that $r=0$ or $s=0$. Then we have $t=t'=0$.

 Assume that $s=0$. Then $\la^R=\mu^R=\emptyset$ and we have
$$\prod_{i=1}^{r} (1+\cont(\la^L,i) z)
=\prod_{i=1}^{r} (1+\cont(\mu^L,i) z).
$$
 It follows that the multisets of contents of $\la^L$ and $\mu^L$  are the same.
 Thus $\la^L= \mu^L $.
The proof for the case $r=0$ is the same.

{\bf Case (2), (3) : } Assume that $\delta \notin \Z$ or $|\delta| > r+s-2$.
We may further assume that $r>0$ and $s>0$ so that $r+s-2 \ge 0$.
Now we have
\eq
&\prod_{i=1}^{r-t} (1+\cont(\la^L,i) z)
=\prod_{i=1}^{r-t'} (1+\cont(\mu^L,i) z)  \quad \text{and} \nonumber \\ 
&\prod_{j=r+t+1}^{r+s} (1-(\cont(\la^R,j)+\delta) z)
=\prod_{j=r+t'+1}^{r+s} (1-(\cont(\mu^R,j)+\delta) z). \label{eq:lower}
\eneq
Indeed if $\delta \notin \Z$ then it is trivial. Otherwise, we have
\eqn
&&| \cont(\la^L,i)+\cont(\la^R,j) |
\le (r-t-1) + (s-t-1) \le  r+s-2 < |\delta|, \\
&&| \cont(\mu^L,i)+\cont(\mu^R,j) |
\le (r-t'-1) + (s-t'-1) \le  r+s-2  < |\delta|.
\eneqn

It follows that  the multiset of \emph{nonzero} contents of $\la^L$ is the same as  the one of $\mu^L$
 and the multiset of  contents of $\la^R$ \emph{which are not equal to $-\delta$} is the same as the one of $\mu^R$.

Note that
\eqn
\set{j}{\cont(\la^R,i)=-\delta}= \set{j}{\cont(\mu^R,j)=-\delta} =\emptyset.
\eneqn
Indeed it it trivial if $\delta \notin \Z$ and otherwise we have
\eqn
&&|\cont(\la^R,i)|\le s-t-1\le r+s-t-2 \le  r+s-2 < |\delta|, \\
&&|\cont(\mu^R,i)|\le s-t'-1\le r+s-t'-2 \le  r+s-2  < |\delta|.
\eneqn
Hence, observing the degree of \eqref{eq:lower}, we have $t=t'$.
Thus the multiset of contents of $\la^L$ is the same as  the one of $\mu^L$
 and the multiset of contents of $\la^R$ is the same as the one of $\mu^R$.
 It follows that $\la=\mu$.

{\bf Case (4) : } Assume $\delta =0$ and $(r,s) \in \{(1,2),(2,1),(1,3),(3,1)\}$.
If $(r,s)=(3,1)$, then we have
\eqn
\set{(c(\la,i))_{1\le i  \le4} }{\la \in \dot{\Lambda}_{3,1}} =\Big\{(0,1,2,0), (0,1,-1,0), (0,-1,-2,0), (0,1,0,0), (0,-1,0,0) \Big\}.
\eneqn
It is easy to observe that \eqref{eq:contents rational function} holds if and only if $\la=\mu$.

The cases $(r,s) \in \{(1,2),(2,1),(1,3)\}$ can be checked in a  similar way.
\QED

\vskip1em
\section{Center of the walled Brauer algebra}

By Proposition \ref{prop:belongtocenter}, the power sum supersymmetric polynomials in the Jucys-Murphy elements belong to
the center of $B_{r,s}(\delta)$. Hence they act by  scalar multiplications on a cell module by Lemma \ref{lem:one dimensional}.
More precisely, we have
\begin{prop} \label{prop:action on a cell module}
For $\la \in \Lambda_{r,s}^t$ and for $k \ge 0$, we have
\eqn
p_k(L_1,\ldots,L_r,L_{r+1},\ldots,L_{r+s}) = \sum_{i=1}^{r-t} \cont(\la^L,i)^k +  (-1)^{k+1} \sum_{i=1}^{s-t}(\cont(\la^R,i)+\delta)^k
\eneqn
on  the cell module  $C_{r,s}(\la)$.
\end{prop}
When $k=1$,  the above  can be obtained from \cite[Lemma 2.3]{BS} directly, because $z_{r,s}$ in \cite[Lemma 2.3]{BS} is
the same as $L_1+ \cdots +L_{r+s} -s \delta$.

To prove Proposition \ref{prop:action on a cell module}, we need the following lemma.
We use the same technique in \cite[Lemma 2.3]{BS}. For reader's convenience, we  include a proof.

\begin{lemma}  \label{lem:relations between tau and L_i}
 For each $1 \le t \le \min(r,s)$, set
 \eq \label{eq:taut}
\tau_t :=e_{r,r+1} e_{r-1,r+2} \cdots e_{r-t+1, r+t} \in B_{r,s}(\delta).
 \eneq
 We have
  \begin{enumerate}[\rm(1)]
    \item  $ (L_r +L_{r+1}) \tau_t =(L_{r-1}+L_{r+2}) \tau_t  = \cdots = (L_{r-t+1} +L_{r+t})  \tau_t  =0 ,$
    \item $ \Big(- \sum_{i=r-t+1}^r e_{i,j} + \sum_{i=r+1}^{r+t} (i, j) \Big)  \tau_t =0$ for any $ j  \ge r+t+1$.
  \end{enumerate}
\end{lemma}

\Proof
(1) Let $I:= \{1, \ldots, r-t \}, \  J=\{r+t+1, \ldots, r+s\}$ and $K:= \{r-t+1, \ldots, r \}$.
For $a \in K$, let $\widetilde{a}:= 2r+1-a$. We shall show that
$$(L_a+L_{\widetilde{a}})\tau_t =0    \quad \quad \text{         for all $a \in K$}.$$

We can obtain the following relations using Lemma \ref{lem:relations between horizontal lines and vertical lines}:
\begin{enumerate}[{\rm(1)}]
\item $((i,a)-e_{i,{\widetilde{a}}})e_{a,\widetilde{a}}=0$ for $i \in I,$
\item $((i,a)-e_{i,\widetilde{a}})e_{i,\widetilde{i}}e_{a,\widetilde{a}} =0$ for $i \in K, i <a,$
\item $(-e_{a,\widetilde{a}} +\delta) e_{a,\widetilde{a}}=0,$
\item $(-e_{i,\widetilde{a}} +(\widetilde{i},\widetilde{a}))e_{i,\widetilde{i}}e_{a,\widetilde{a}}=0$ for $i \in K, i>a.$
\end{enumerate}
More precisely, we can check that
\vskip3mm
\begin{enumerate}[{\rm(1)}]
\item $((i,a)-e_{i,{\widetilde{a}}})e_{a,\widetilde{a}}=e_{i,\widetilde{a}}(i,a)-e_{i,\widetilde{a}}(i,a)=0,$
\item $((i,a)-e_{i,\widetilde{a}})e_{i,\widetilde{i}}e_{a,\widetilde{a}}
=e_{a, \widetilde{i}}(i,a)e_{a,\widetilde{a}}-e_{i,\widetilde{a}}(\widetilde{i}, \widetilde{a})e_{a,\widetilde{a}}
=e_{a, \widetilde{i}}e_{i,\widetilde{a}}(i,a)-e_{i,\widetilde{a}}e_{a,\widetilde{i}}(\widetilde{i},\widetilde{a})=0,$
\item $(-e_{a,\widetilde{a}} +\delta) e_{a,\widetilde{a}}=-\delta e_{a,\widetilde{a}}+\delta e_{a,\widetilde{a}}=0,$
\item $(-e_{i,\widetilde{a}} +(\widetilde{i},\widetilde{a}))e_{i,\widetilde{i}}e_{a,\widetilde{a}}
=-e_{i,\widetilde{a}}(\widetilde{a},\widetilde{i})e_{a,\widetilde{a}}+e_{i,\widetilde{a}}(\widetilde{i},\widetilde{a})e_{a,\widetilde{a}}
=-e_{i,\widetilde{a}}e_{a,\widetilde{i}}(\widetilde{a},\widetilde{i})+e_{i,\widetilde{a}}e_{a,\widetilde{i}}(\widetilde{i},\widetilde{a})=0.$
\end{enumerate}
\vskip3mm

We separate $L_a= \sum_{i \in I} (i,a) + \sum_{i < a, i\in K} (i,a)$, and
$$L_{\widetilde{a}}= - \sum_{i \in I} e_{i,\widetilde{a}} - \sum_{i<a, i\in K} e_{i,\widetilde{a}} -e_{a,\widetilde{a}}
-\sum_{i>a, i \in K} e_{i,\widetilde{a}}+\sum_{i>a, i \in K}
(\widetilde{i},\widetilde{a}) +\delta.$$
Note that all the factors in \eqref{eq:taut}  are commuting to each other by Lemma \ref{lem:relations between horizontal lines and vertical lines} (3).
Using the relations (1)-(4), we have $(L_a +L_{\widetilde{a}})\tau_t=0.$

(2) For  $i\in K$ and $j \in J$, we have
\eqn
(-e_{i, j} +(\widetilde{i} , j) ) e_{i,\widetilde{i}} =-e_{i, j}( j, \widetilde{i}) +e_{i,j} (\widetilde{i}, j)=0.
\eneqn
Thus 
we have
\eqn
(-e_{i,j} +(\widetilde{i} , j) )\tau_t=0.
\eneqn
It follows that
 \eqn
 \Big(-\sum_{i=r-t+1}^r e_{i,j}+\sum_{i=r+1}^{r+t} (i,j) \Big) \tau_t
=\sum_{i=r-t+1}^r (-e_{i,j}+(\widetilde{i},j)) \tau_t  =0.
\eneqn
\QED

Now we prove Proposition \ref{prop:action on a cell module}.

\Proof
 In $C_{r,s}(\lambda)$, the element $\tau \otimes v$ generates $C_{r,s}(\lambda)$
as a $B_{r,s}(\delta)$-module, where $\tau$ is the image of $\tau_t= e_{r,r+1} e_{r-1, r+2} \cdots e_{r-t+1, r+t} $
 $(t \ge 1)$ or $\one$ ($t=0$)  in $J^t_{r,s}/J^{t+1}_{r,s}$ and $v$ is a non-zero vector in $S(\la^L) \boxtimes S(\la^R)$.
Since $p_k(L_1,\ldots,L_{r+s})$ is central, it is enough to show that
\eqn
&& p_k(L_1,\ldots,L_r, L_{r+1},\ldots,L_{r+s})(\tau \otimes v) \\
&& \hskip7em = \Big( \sum_{i=1}^{r-t}\cont(\la^L,i)^k +  (-1)^{k+1} \sum_{i=1}^{s-t}(\cont(\la^R,i)+\delta)^k  \Big)  (\tau \otimes v),
\eneqn
 for all $k \ge 1$.
By Lemma \ref{lem:relations between tau and L_i} (1), we know that
\eqn
(L_{r-t+1}^k + \cdots +L_r^k +(-1)^{k+1}(L_{r+1}^k + \cdots + L_{r+t}^k))\tau_t=0.
\eneqn
 Hence it is enough to show that
\eqn
\Big(\sum_{j=1}^{r-t} L_j^k \Big) \big(\tau \otimes v \big)
=   \Big( \sum_{i=1}^{r-t} \cont(\la^L,i)^k \Big) \tau\otimes v \quad \text{and,} \\
\Big(\sum_{j=r+t+1}^{r+s} L_j^k \Big) \big(\tau \otimes v \big)
=\Big(\sum_{j=1}^{s-t}  (\cont(\la^R, j )+\delta)^k \big)\tau\otimes v.
\eneqn

Note that
$(L_1^k + \cdots +L_{r-t}^k)\tau=\tau (L_1^k + \cdots +L_{r-t}^k)$
and
$L_1^k + \cdots +L_{r-t}^k \in \C[\sym_{r-t}] \otimes \one \subset \C[\sym_{r-t}] \otimes \C[\sym_{s-t}]$.
Moreover, $L_1^k + \cdots +L_{r-t}^k$ corresponds to a symmetric polynomial in the Jucys-Murphy elements of
$\C[\sym_{r-t}]$ under the  isomorphism
 $\langle s_1,\ldots, s_{r-t-1} \rangle \simeq \C [\sym_{r-t}]$.  
Thus we obtain
\eqn
&&(L_1^k + \cdots +L_{r-t}^k)\tau\otimes v
=\tau(L_1^k + \cdots +L_{r-t}^k)\otimes v \\
&&=\tau \otimes  (L_1^k + \cdots +L_{r-t}^k) v
=  \tau\otimes \big( \sum_{i=1}^{r-t} \cont(\la^L,i)^k \big)  v,
\eneqn
where the last equality follows from the fact that a symmetric polynomial in Jucys-Murphy elements of a symmetric group acts on a simple module $S(\mu)$ associated with a partition $\mu$ by the scalar multiplication, which is given by the evaluation of the polynomial at the contents of $\mu$ (\cite{Murphy81}). See also \cite[Theorem 1.1]{DG}.

When $1 \le i \le r-t$, and $r+t+1 \le j \le r+s$, we get $e_{i,j} \tau =0$ on $J_{r,s}^t/J^{t+1}_{r,s}$ because $e_{i,j}$ makes another horizontal strand.
Combining this and Lemma \ref{lem:relations between tau and L_i} (2), we have
\eqn
L_j \tau = \Big(\sum_{i=r+t+1}^{j-1} (i, j) + \delta \Big) \tau = \tau \Big(\sum_{i=r+t+1}^{j-1} (i, j) + \delta \Big) \ \quad  \  \text{for } r +t+1 \le j \le r+s
\eneqn
on $J_{r,s}^t/J^{t+1}_{r,s}$.
Note that we used the fact that $\tau$ commutes with $(p,q)$, if $p,q \ge r+t+1$.
By  repeating this procedure,  we obtain
\eqn
L_j^k \tau =  \tau \Big(\sum_{i=r+t+1}^{j-1} (i, j) + \delta \Big)^k \ \quad  \  \text{for } r +t+1 \le j \le r+s
\eneqn
on $J_{r,s}^t/J^{t+1}_{r,s}$.
Now we have
\eqn
\Big(\sum_{j=r+t+1}^{r+s} L_j^k \Big) \big(\tau \otimes v \big)
= \sum_{j=r+t+1}^{r+s}\tau\Big(\sum_{i=r+t+1}^{j-1} (i, j) + \delta \Big)^k \otimes v
= \tau \otimes \Big( \sum_{j=r+t+1}^{r+s}\Big( \sum_{i=r+t+1}^{j-1} (i, j) + \delta \Big)^k \Big) v,
\eneqn
because $\sum_{i=r+t+1}^{j-1} (i, j) + \delta
\in \one \otimes \C[\sym_{s-t}] \subset \C[\sym_{r-t}] \otimes \C[\sym_{s-t}]$.
Since the element 
$$\sum_{j=r+t+1}^{r+s}\Big(\sum_{i=r+t+1}^{j-1} (i, j) + \delta\Big)^k$$
 is a symmetric polynomial
in the Jucys-Murphy elements of
the subalgebra $\one \otimes \C [\sym_{s-t}]$  under the isomorphism $\langle s_{r+t+1},\ldots, s_{r+s-1} \rangle \simeq \C [\sym_{s-t}],$
it follows again by \cite{Murphy81} (see also \cite[Theorem 1.1]{DG}) that 
\eqn
\Big( \sum_{j=r+t+1}^{r+s}\Big( \sum_{i=r+t+1}^{j-1} (i, j) + \delta \Big)^k \Big) v
= \Big( \sum_{j=1}^{s-t}  (\cont(\la^R,  j  )+\delta)^k\Big)  v.
\eneqn
Therefore, we obtain the desired assertion.
\QED

\begin{corollary} \label{cor:constant}
Let $f$ be a supersymmetric polynomial in $S_{r,s}[x;y]$.
Then we have
\eqn
f(L_1,\ldots,L_r,L_{r+1},\ldots,L_{r+s}) = f(c(\la,1),\ldots,c(\la,r),c(\la,r+1), \ldots, c(\la ,r+s))
\eneqn
on $C_{r,s}(\la)$ for every $\la \in \La_{r,s}$.
\end{corollary}
\Proof
It follows immediately from the above proposition and the fact that $S_{r,s}[x;y]$ is generated by the power sum supersymmetric polynomials (\cite{Stem}).
\QED

The proof of the following lemma is  identical to the argument  in  \cite[Theorem 3.3]{Li}.
We reproduce it here in a slightly more general setting for reader's  convenience.
\begin{lemma} \rm{(}\cite[Theorem 3.3]{Li}\rm{)} \label{prop:YanboLi}
Let $S$ be a $\C$-subalgebra of the polynomial ring $\C[X_1,\ldots,$ $X_m]$ and
  let
  \eq \label{eq:sequences}
  (k_{11},\ldots,k_{1m}), \ldots,  (k_{n1},\ldots, k_{nm})
  \eneq
 be  $n$ sequences of elements in $\C$ for some positive integer $n$.
Assume that
\eq \label{eq:assump} &&\text{for each $1 \le i  \neq j \le n$, there exists an element $p$ in $S$}  \\
&&\text{\qquad such that $p(i) \neq p(j)$, \nonumber
}\eneq
where $p(i)$ denotes the value $p(k_{i1},\ldots,k_{im})$.

Then there exists a family of elements
$p_1,\ldots, p_n$ in $S$ such that
  \eqn
  \begin{array}{|cccc|}
    p_1(1) & p_1(2) & \cdots & p_1(n) \\
    p_2(1) & p_2(2) & \cdots & p_2(n) \\
    \cdots & \cdots & \cdots & \cdots \\
    p_n(1) & p_n(2) & \cdots & p_n(n)
  \end{array}
  \neq  0.
  \eneqn
\end{lemma}

\begin{proof}
  We will proceed by induction on $n$.
Let $(k_{1},k_{2},\ldots,k_{m})$ be a sequence of elements in $\C$.
There exists an element $p$ such that $p(k_{1},k_{2},\ldots,k_{m}) \neq 0$.
For example, we can take  a nonzero constant polynomial as $p$.

Now assume that $n >1$ and that the assertion holds for all $1 \le i \le n-1$.
 Consider the first  $n-1$ sequences
  \eqn
  (k_{11},\ldots,k_{1m}), \ldots,  (k_{n-1 1},\ldots, k_{n-1 m})
  \eneqn
of \eqref{eq:sequences}.
Then,  for $1\le i \neq  j \le n-1$ there exists a polynomial $p \in
S$ such that
$p(i) \neq p(j)$, by the assumption \eqref{eq:assump}.
Now, by the induction hypothesis,  we assume that there exists a family of  polynomials
$p_1,\ldots, p_{n-1}$ in $S$ such that
  \eq \label{eq:upper}
  \begin{array}{|cccc|}
    p_1(1) & p_1(2) & \cdots & p_1(n-1) \\
    p_2(1) & p_2(2) & \cdots & p_2(n-1) \\
    \cdots & \cdots & \cdots & \cdots \\
    p_{n-1}(1) & p_{n-1}(2) & \cdots & p_{n-1}(n-1)
  \end{array}
  \neq  0.
  \eneq
Applying elementary row operations to the above matrix,
we may further assume that
\eq \label{eq:pij=0}
p_i (j) =0\ (1 \le j < i \le n-1), \quad \text{and}   \quad p_i(i)=1  \ (1 \le i \le n-1)
\eneq
(see, \cite[Lemma 3.5]{Li}). In particular, the determinant of the above matrix is $1$.

Suppose that
for every element $p \in S$ we have
  \eqn
  d(p):=\begin{array}{|ccccc|}
    p_1(1) & p_1(2) & \cdots & p_1(n-1) & p_1(n)\\
    p_2(1) & p_2(2) & \cdots & p_2(n-1) & p_2(n)\\
    \cdots & \cdots & \cdots & \cdots &\cdots  \\
    p_{n-1}(1) & p_{n-1}(2) & \cdots & p_{n-1}(n-1) &p_{n-1}(n) \\
    p(1) & p(2) & \cdots & p(n-1) & p(n)
  \end{array}
  =0.
  \eneqn
Then,
by \eqref{eq:upper}, \eqref{eq:pij=0} and the determinant expansion by minors,
we have
\eqn
p(n)=k_1p(1)+k_2p(2)+\cdots+k_{n-1}p(n-1),
\eneqn
for some $k_j \in \C$, which is independent of $p$  for $1\le j \le n-1$.
Note that $p_{n-1}p$  is also in $S$ so that  we have
\eqn
p_{n-1}p(n)=k_1p_{n-1}p(1)+k_2p_{n-1}p(2)+\cdots+k_{n-1}p_{n-1}p(n-1)
=k_{n-1}p(n-1),
\eneqn
since $p_{n-1}(j)=0$ for $1 \le j \le n-2$ and $p_{n-1}(n-1)=1$.

Assume that  $p_{n-1}(n)\neq 0$.
Then we have
$$p(n)=\dfrac{k_{n-1}}{p_{n-1}(n)}p(n-1),
$$
for all polynomial $p \in S$.
Taking a nonzero constant  polynomial $p \in S$, we have
$\dfrac{k_{n-1}}{p_{n-1}(n)}=1$ so that
$p(n)=p(n-1)$
for all  polynomials $p$ in $S$.
It is a contradiction to the assumption \eqref{eq:assump} for $i=n-1, j=n$.

Thus we have $p_{n-1}(n)=0$.
It implies that
$k_{n-1}p(n-1)=0$. Taking a nonzero constant polynomial $p$, we have
$k_{n-1}=0$.
Repeating this process similarly with $p_j$ instead of $p_{n-1}$, we have $k_j=0$ for $j=1,\ldots, n-1$ and hence
$p(n)=0$. But it is impossible, since $p$ is arbitrary.
Thus we conclude that there exists an element  $p$ in $S$ such that $d(p) \neq 0$.
\end{proof}

Now we present our main theorem.

\begin{theorem} \label{thm:main}
If the walled Brauer algebras $B_{r,s}(\delta)$ is semisimple, then
  the supersymmetric polynomials in $L_1,\ldots,L_{r+s}$ generate the center of $B_{r,s}(\delta)$.
\end{theorem}

\Proof
 By Lemma \ref{lem:separation}, we can apply the Proposition \ref{prop:YanboLi} to the case $S=S_{r,s}[x;y]$ with the sequences
$$\set{(c(\la,i))_{1\le i \le r+s}}{\la \in \dot{\Lambda}_{r,s}}.$$
Thus we obtain a set of supersymmetric polynomials
$\set{p_\la \in S_{r,s}[x;y]}{\la \in \dot{\Lambda}_{r,s}}$ such that the matrix
 \eqn
 \big( (p_\la(\mu))\big)_{\la, \mu \in \dot{\Lambda}_{r,s}}
  \eneqn
is nonsingular, where $p_\la(\mu):=p_\la(c(\mu,1),\ldots, c(\mu,r+s))$.

Assume that there is a family of complex numbers $\set{a_\la \in \C }{\la \in \dot{\Lambda}_{r,s}}$ such that
$$\sum_{\la} a_\la p_\la(L_1,\ldots,L_{r+s})=0.$$
By Corollary \ref{cor:constant}, we have
$$\sum_{\la} a_\la p_\la(\mu)=0 \quad \text{for all } \ \mu \in \dot{\Lambda}_{r,s}.$$
Hence $a_\la=0$ for all $\la \in \dot{\Lambda}_{r,s}$ so that
$$p_\la(L_1,\ldots,L_{r+s}) \quad (\la \in \dot{\Lambda}_{r,s})$$
are linearly independent.

On the other hand, if $B_{r,s}(\delta)$ is semisimple, the dimension of the center is the same as the number of the isomorphism classes of  simple modules, and hence it is identical to the cardinality of $\dot{\Lambda}_{r,s}$.
Thus we conclude that
$\set{p_\la(L_1,\ldots,L_{r+s})}{\la \in \dot{\Lambda}_{r,s}}$ is a basis of the center, as desired.
\QED

\vskip1em
\section{Gelfand-Zetlin subalgebras}
In this section, we assume that  $B_{r,s}(\delta)$  is semisimple.
For the sake of simplicity, we assume further that $r \ge s$.
For $1 \le a \le r+s$, let $B_a$ be the subalgebra of $B_{r,s}(\delta)$
generated by
\eqn
B_a=
\begin{cases}
\langle s_1,\ldots, s_{a-1}  \rangle & \text{if} \ 1 \le a \le r, \\
\langle s_1,\ldots, s_{r-1}, e_{r,r+1}, s_{r+1} \ldots s_{a-1} \rangle  & \text{if} \ r+1 \le a \le r+s.
\end{cases}
\eneqn
Set $B_0:=\C$. Then
we have  a tower of subalgebras
\eqn
 \C = B_0 \subset B_1 \subset \cdots \subset B_{r+s-1} \subset B_{r+s} =B_{r,s}(\delta).
\eneqn
This tower of subalgebras are compatible with the Jucys-Murphy elements in the following sense:
for each $1 \le a \le r+s$, $L_1,\ldots, L_a$ are contained in $B_a$  and they are exactly the Jucys-Murphy elements of $B_a$ under the isomorphism
\eqn
B_a \simeq
\begin{cases}
 B_{a,0}(\delta)& \text{if} \ 1 \le a \le r, \\
 B_{r,a-r}(\delta) & \text{if} \ r+1 \le a \le r+s.
\end{cases}
\eneqn
The above isomorphisms are obtained by checking that the generators of $B_a$ produce
all the $(r,s)$-walled Brauer diagrams with $r+s-a$ vertical strands
connecting the $k$-th vertex on the top row with the $k$-th vertex on the bottom row for each $a+1 \le k \le r+s$.
In particular,
$B_a$ is semisimple for all $0 \le a \le r+s$.
Set
\eqn
\Lambda_a:=
\begin{cases}
\{ \emptyset \} &  \text{if} \ a=0,  \\
\Lambda_{a,0} & \text{if} \ 1 \le a \le r, \\
\Lambda_{r,a-r} & \text{if} \ r+1 \le a \le r+s.
\end{cases}
\eneqn
 For $\lambda \in \Lambda_a$, we define
\eqn
C_a(\lambda):=
\begin{cases}
 \C  &  \text{if} \ a=0,  \\
C_{a,0}(\lambda) & \text{if} \ 1 \le a \le r, \\
C_{r,a-r}(\lambda) & \text{if} \ r+1 \le a \le r+s.
\end{cases}
\eneqn
For $1 \le a\le r$ and $\la \in \Lambda_{a}$, we have
\eq \label{eq:rest1}
\Res^{B_{a}}_{B_{a-1}}  C_{a}(\la)  \simeq \bigoplus_\mu  C_{a-1}(\mu),
\eneq
where the sum runs over the weights $\mu \in \Lambda_{a-1}$ such that
 the skew Young diagram $[\la^L] / [\mu^L]$ consists of a single box.
 This is nothing but the branching rule for symmetric groups.
For $r+1 \le a\le r+s$ and $\la \in \Lambda_a$, it is shown in \cite[Theorem 3.16]{Halverson} that
\eq\label{eq:rest2}
\Res^{B_{a}}_{B_{a-1}}  C_{a}(\la) \simeq \bigoplus_\mu C_{a-1}(\mu),
\eneq
where the sum runs over the weights $\mu \in \Lambda_{a-1}$ such that
either the  diagram $[\mu^L] / [\la^L]$ consists of a single box, or the diagram $[\la^R] / [\mu^R]$ consists of a single box (see also \cite[Corollary 3.6]{CDDM}).
In particular, for $1 \le a \le r+s$, the restriction of any simple module of $B_{a}$ to $B_{a-1}$ is multiplicity-free so that the above decompositions are canonical.
From now on, we identify $C_{a-1}(\mu)$ in the right hand side  in \eqref{eq:rest1}, \eqref{eq:rest2} with the unique  simple submodule of $C_a(\la)$ which is isomorphic to it.

Let $\B$ be the \emph{branching graph}  
 of $B_{r,s}(\delta)$:
 the set of vertices is given by
$\bigsqcup_{a=0}^{r+s} \Lambda_a$  and the two vertices $\mu \in \Lambda_a$ and $\la \in \Lambda_{a+1}$ are joined by an arrow from $\mu$ to $\la$ if and only if
$\Hom_{B_{a}}( C_a(\mu) , \Res_{B_a}^{B_{a+1}}  C_{a+1}(\la) ) \neq 0$.
Iterating the restrictions, we obtain a canonical decomposition of a simple $B_{r,s}$-module $ C_{r,s}(\lambda) $ into a direct sum of simple $B_0$-module, i.e., $1$-dimensional subspaces
\eqn
C_{r,s}(\la)  =\bigoplus_{T} V_T,
\eneqn
where the sum runs over all the paths $T$ in $\B$ from $\emptyset$ to $\la$.
Taking a non-zero vector $v_T$ for each path $T$ in $\B$ from $\emptyset$ to $\la$, we obtain a basis
$\set{v_T}{\text{ $T$ : a path in $\B$ from $\emptyset$ to $\la$}}$ of $C_{r,s}(\lambda) $, called
the \emph{Gelfand-Zetlin basis} (shortly, GZ-basis) of $C_{r,s}(\lambda)$.

Let $A_{r,s}$ be the subalgebra of $B_{r,s}(\delta)$ generated by $Z(B_1),\ldots,Z(B_{r+s})$, where $Z(B_a)$ denotes the center of $B_a$.
We call $A_{r,s}$ the \emph{Gelfand-Zetlin subalgebra} (shortly, the GZ-subalgebra) of $B_{r,s}(\delta)$.
Note that $A_{r,s}$  is commutative.

\begin{prop} \label{prop:generate}
The GZ-subalgebra $A_{r,s}$ is generated by $L_1, \ldots, L_{r+s}$.
\end{prop}
\Proof
By Proposition \ref{prop:belongtocenter}, we have  $L_1 +\cdots +L_{a-1} \in Z(B_{a-1})$ and $L_1+ \cdots +L_{a-1}+L_a \in Z(B_a)$. It follows that $L_a \in A_{r,s}$ for $1 \le a \le r+s$.
On the other hand, by Theorem \ref{thm:main}, each $Z(B_a)$ is contained in the subalgebra of $B_a$ generated by $L_1,\ldots,L_a$, and hence $A_{r,s}$ is generated by $L_1, \ldots, L_{r+s}$.
\QED

By Wedderbun-Artin theorem, we have a $\C$-algebra isomorphism
\eq \label{eq:matrix form of the walled Brauer algebra}
B_{r,s}(\delta) \simeq \bigoplus_{\la \in \Lambda_{r,s}} \End_{ \C }( C_{r,s}(\lambda) ).
\eneq
For each $\la \in \Lambda_{r,s}$, we identify the algebra $\End_{ \C }( C_{r,s}(\lambda) )$ with the algebra of $\dim_\C  C_{r,s}(\lambda)  \times \dim_\C  C_{r,s}(\lambda) $  matrices over $\C$, by taking a GZ-basis of $C_{r,s}(\lambda)$.
Then we have the following proposition whose proof is identical to the one of \cite[Proposition 1.1]{OV2}.
\begin{prop}  \label{prop:maximal}
The GZ-subalgebra $A_{r,s}$ is identified with the set of all the diagonal matrices. In particular, $A_{r,s}$ is a maximal commutative subalgebra of $B_{r,s}(\delta)$.
\end{prop}

For a path $T$ in $\B$ given by
$$
T=\emptyset \to \la_1 \to \cdots \to \la_{r+s-1} \to \la_{r+s}, \quad
(\la_a \in \Lambda_a),
$$
set
\eqn
c_T(i) := \begin{cases}
\text{content of }[\la_i^L] /[\la_{i-1}^L]
& \text{if} \ 1 \le i \le  r, \\
-\text{content of }[\la_{i-1}^L] /[\la_{i}^L]
& \text{if} \ r+1 \le i \le  r+s, \ [\la_{i-1}^L] /[\la_{i}^L] \ \text{is a single box}  \\
\text{content of }[\la_i^R] /[\la_{i-1}^R] + \delta,
& \text{if} \ r+1 \le i \le  r+s,  \ [\la_i^R] /[\la_{i-1}^R] \ \text{is a single box}.
\end{cases}
\eneqn
The below was shown for the case $1 \le i \le r$  in  \cite{Murphy81} (see also \cite[Theorem 1.1]{DG} and \cite[Section 5]{OV2}).
\begin{prop} \label{prop:action on a GZ basis}
Let $T$ be given as the above.
Then we have
\eqn
L_i \, v_T = c_T(i) v_T \quad (1 \le i \le r+s).
\eneqn
\end{prop}

\Proof
For each $\la \in \bigsqcup_{a=0}^{r+s} \Lambda_a$, set
\eqn
c(\la):=\sum_{j=1}^{|\la^L |} \cont(\la^L,j)+\sum_{j=1}^{|\la^R|}(\cont(\la^R,j)+\delta).
\eneqn
Observe that
\eqn
c(\la_i)-c(\la_{i-1})=c_T(i)  \quad \text{for } \ 1 \le i \le r+s.
\eneqn
On the other hand, by Proposition \ref{prop:action on a cell module}, we know
\eqn
(L_1 + \cdots + L_i) v_T
=c(\la_i) v_T \quad \text{for } \ 1 \le i \le r+s.
\eneqn
In particular, we have
\eqn
L_i v_T =(c(\la_i)-c(\la_{i-1})) v_T=c_T(i) v_T,
\eneqn
as desired.
\QED

The following is an analogue of \cite[Theorem 2.1]{Murphy}. See also \cite[Definition 3.1]{Mathas}.
\begin{prop}
For each $1\le i \le r+s$, set
$$\mathcal C(i):=\set{c_T(i)}{\text{ $T$ : a path in $\B$ from $\emptyset$ to $\la$
 for some $\la \in \Lambda_{r+s}$}}.$$
The following elements  form a complete set of primitive orthogonal idempotents of $B_{r,s}(\delta)$:
\eqn
I_T:=  \prod_{i=1}^{r+s} \prod_{\overset{c \in \mathcal C(i)}{c \neq c_T(i)}}
\dfrac{L_i-c}{c_T(i)-c}.
\eneqn
\end{prop}
\Proof
Combining Proposition \ref{prop:generate} and Proposition  \ref{prop:maximal}, we know that $T=T'$ if and only if $c_T(i)=c_{T'}(i)$ for all $1 \le i \le r+s$ (see \cite[Remark 1.2]{OV2}).
Hence we have
\eqn
I_T v_{T'} = \prod_{i=1}^{r+s} \prod_{\overset{c \in \mathcal C(i)}{c \neq c_T(i)}}
\dfrac{c_{T'}(i)-c}{c_T(i)-c} = \delta_{T, T'} v_{T'},
\eneqn
as desired.
\QED
\begin{remark}
The above can be proved by checking that
a path $T$ in $\B$ is uniquely determined by $(c_{T}(i))_{1 \le i \le r+s}$, whenever the triple $(r,s, \delta)$ belongs to  one of the cases in Proposition \ref{prop:semisimplicity}.
\end{remark}

\vskip1em
\section{Blocks of walled Brauer algebra}
In this section, we return to the cases in which $r,s$ and $\delta$ are arbitrarily chosen.
We say that two simple $B_{r,s}(\delta)$-modules $D_{r,s}(\la)$ and $D_{r,s}(\mu)$ \emph{belong to the same block}
 if there exists a sequence of simple $B_{r,s}(\delta)$-modules $D_{r,s}(\la)= D_1, D_2, . . . , D_k = D_{r,s}(\mu)$ such that either $\Ext^1_{B_{r,s}(\delta)}(D_i, D_{i+1})\neq 0$ or $ \Ext^1_{B_{r,s}(\delta)}(D_{i+1}, D_{i}) \neq 0$, for  all  $1 \le i < k$.
  For each $\la \in \dot{\Lambda}_{r,s}$, every element $z$ in the center acts on $D_{r,s}(\la)$ by a scalar multiple, say $\psi_\la(z)$. The assignment $z \mapsto \psi_\la(z)$ defines a $\C$-algebra homomorphism
 $$\psi_\la : Z(B_{r,s}(\delta)) \to \C,$$
and we call $\psi_\la$ the \emph{central character afforded by $D_{r,s}(\la)$}.
Two simple modules $D_{r,s}(\la)$ and $D_{r,s}(\mu)$ belong to the same block if and only if the central characters
$\psi_\la$ and $\psi_\mu$ are identical (see, for example, \cite[Chapter I, Proposition 10.15]{Karp}).
Note that the scalar multiplication on $C_{r,s}(\la)$ induced  by a central element $z$ equals  to $\psi_\la(z)$, because  $D_{r,s}(\la)$ is a quotient of $C_{r,s}(\la)$.

 We will say that $(\la^L,\la^R)$ and $(\mu^L,\mu^R)$ are \emph{$\delta$-balanced} if there is a pairing of the boxes in $[\la^L]/([\la^L] \cap [\mu^L])$ with those in $[\la^R]/([\la^R] \cap [\mu^R])$ and a pairing of the boxes in $[\mu^L]/([\la^L] \cap [\mu^L])$ with those in $[\mu^R]/([\la^R] \cap [\mu^R])$ such that the contents of each pair sum to $-\delta$.
In \cite{CDDM}, the blocks of $B_{r,s}(\delta)$ are classified as follows:

\begin{prop} {\rm (}\cite[Corollary 7.7]{CDDM} {\rm )}
Two simple modules $D_{r,s}(\la)$ and $D_{r,s}(\mu)$ are in the same block of $B_{r,s} (\delta)$
 if and only if the weights $\la$ and $\mu$ are $\delta$-balanced.
\end{prop}

The following lemma shows that the notion of $\delta$-balanced weights is closely related to the
contents evaluation of supersymmetric polynomials.
\begin{lemma} \label{lem:equiv_balanced}
Two weights $\la=(\la^L,\la^R)$ and $\mu=(\mu^L,\mu^R)$ are \emph{$\delta$-balanced} if and only if
$$p((c(\la,i))_{1\le i \le r+s}) = p((c(\mu,i))_{1\le i \le r+s})$$
for every supersymmetric polynomial $p \in S_{r,s}[x;y]$.
\end{lemma}
\Proof
Recall that for $\la,\mu \in \Lambda_{r,s}$ we have
$$p((c(\la,i))_{1\le i \le r+s}) = p((c(\mu,i))_{1\le i \le r+s})$$
for every supersymmetric polynomial $p \in S_{r,s}[x;y]$
if and only if the equation
\eqref{eq:contents rational function}
\eqn
\dfrac{\prod_{i=1}^{r-t} (1+\cont(\la^L,i) z)}{\prod_{j=r+t+1}^{r+s} (1-(\cont(\la^R,j)+\delta) z)}
=\dfrac{\prod_{i=1}^{r-t'} (1+\cont(\mu^L,i) z)}{\prod_{j=r+t'+1}^{r+s} (1-(\cont(\mu^R,j)+\delta) z)}
\eneqn
holds.
It is equivalent to saying that
\eqn
&&\sharp\set{u \in [\la^L]}{\text{content of the box $u$ = $k$}}
-\sharp\set{u \in [\la^R]}{\text{content of the box $u$ = $-\delta-k$}}\\
=&&\sharp\set{u \in [\mu^L]}{\text{content of the box $u$ = $k$}}
-\sharp\set{u \in [\mu^R]}{\text{content of the box $u$ = $-\delta-k$}}
\eneqn
for all $k \in \Z$.

It is shown in \cite[Lemma 8.1]{CDDM} that the above condition is equivalent to saying that
$\la$ and $\mu$ are $\delta$-balanced.
\QED

We give an alternative proof of a part of the classification theorem of blocks of $B_{r,s}(\delta)$, appeared in \cite{CDDM}.
\begin{prop} {\rm (}\cite[Corollary 7.3]{CDDM}{\rm)} \label{prop:blocks}
If two simple modules $D_{r,s}(\la)$ and $D_{r,s}(\mu)$ are in the same block of $B_{r,s}(\delta)$, then
the weights $\la$ and $\mu$ are $\delta$-balanced.
\end{prop}

\Proof
We have $\psi_\la=\psi_\mu$. In particular, by Corollary \ref{cor:centerssym}, and Corollary \ref{cor:constant},
we have
$$p((c(\la,i))_{1\le i \le r+s}) = \psi_\la (p(L_1,\ldots,L_{r+s}))
= \psi_\mu(p(L_1,\ldots,L_{r+s}))=p((c(\mu,i))_{1\le i \le r+s})$$
 for every supersymmetric polynomial $p \in S_{r,s}[x;y]$.
Hence $\la$ and $\mu$ are $\delta$-balanced by the above lemma.
\QED

We expect that the following generalization of our main theorem holds.
 Note that it corresponds to \cite[Conjecture 7.4]{SartoriStroppel}, if we take the modification given  in Remark \ref{rem:different} and focus on $B_{r,s}(\delta)$.
\begin{conjecture} {\rm (see also \cite[Conjecture 7.4]{SartoriStroppel})} \label{conj:center}
For every  $\delta \in \C$,
the center of walled Brauer algebra $B_{r,s}(\delta)$ is generated by the supersymmetric polynomials in the Jucys-Murphy elements $L_1,\ldots,L_{r+s}$.
\end{conjecture}

\begin{remark} \label{rem:if the conjecture is true,...}
If the above conjecture is true, then the
central characters $\psi_\la$ and $\psi_\mu$ are identical  for every $\delta$-balanced pair $\la$ and $\mu$. Indeed, the central characters are given by the contents evaluation of supersymmetric polynomials in Jucys-Murphy elements, and by Lemma
\ref{lem:equiv_balanced}, the contents evaluations are identical whenever  $\la$ and $\mu$ are $\delta$-balanced.
In turn, the simple modules  $D_{r,s}(\la)$ and $D_{r,s}(\mu)$ belong to the same block. It would recover  \cite[Corollary 7.7]{CDDM}, the other direction of the classification of the blocks of $B_{r,s}(\delta)$.
\end{remark}

We close the paper with an example  supporting the above conjecture in  small degree.
\begin{example}
Let $r=2, s=2$.
By the method in \cite{Shalile}, we obtain  a  basis of the centralizer  $Z_{B_{2,2}(\delta)}(\C[\sym_2 \times \sym_2])$  of the subalgebra $\C[\sym_2 \times \sym_2]$ in $B_{2,2}(\delta)$. Each element of the basis corresponds to a  \emph{walled generalized cycle types}. For the detail, see \cite[Section 7]{Shalile}.
We enumerate thus obtained elements as follows:
\eqn
&&C_1=\one, \quad C_2=(1,2), \quad C_3=(3,4) \quad C_4=e_{2,3} + e_{1,3} + e_{1,4} + e_{2,4}, \\  &&C_5=e_{1,3} e_{2,4} + e_{1,4} e_{2,3}, \quad C_6=C_2 C_3, \quad C_7=C_2 C_4, \quad C_8= C_3 C_4, \\
&&C_9=C_2 C_5, \quad C_{10}=C_4 C_6.
\eneqn
In particular, we have $\dim Z_{B_{2,2}(\delta)}(\C[\sym_2 \times \sym_2])=10$.
Note that an element in the centralizer  $Z_{B_{2,2}(\delta)}(\C[\sym_2 \times \sym_2])$ is central in $B_{2,2}(\delta)$ if and only if it commutes with the generator $e_{2,3}$.
Now the equation
\eqn
\Big(\sum_{i=1}^{10} a_i C_i \Big) e_{2,3}  - e_{2,3} \Big( \sum_{i=1}^{10} a_i C_i \Big) =0
\eneqn
yields the following system of  $4$ linear equations:
\eqn
&&a_2+a_4+\delta a_7+a_{10}=0,
\quad a_3+a_4+\delta a_8+a_{10}=0, \\
&&a_5+a_7+a_8+\delta a_9=0,
\quad a_6+a_7+a_8+\delta a_{10}=0.
\eneqn

From the above, we obtain a basis of $Z(B_{2,2}(\delta))$:
\eqn
&&\{ B_1:=\one, \ B_2:=C_4-C_2-C_3, \ B_3:=C_8-\delta C_3-C_5-C_6, \ B_4:=C_7-\delta C_2 -C_5-C_6,  \\
&&\   \ B_5:=C_9-\delta C_5, \ B_6=C_{10}-C_3-\delta C_6-C_2 \
\}.
\eneqn
In particular, we have
 $\dim Z(B_{2,2}(\delta))=6$, which is independent from $\delta$ (see also the conjecture on the dimension of the center of the Brauer algebra in \cite[Introduction]{Shalile}).

On the other hand, we obtain the following equalities by direct calculations:
\eqn
 \ p_0(L_1,L_2,L_3,L_4)  &&= \one=B_1, \\
 \ p_1(L_1,L_2,L_3,L_4)  &&= 2\delta C_1 +C_2+C_3-C_4 = 2\delta B_1 - B_2,\\
  \ e_2(L_1,L_2,L_3,L_4) &&=
  (3 \delta^2+1)C_1+ 2 \delta C_2+3 \delta C_3-2 \delta C_4+C_5+C_6-C_8  \\
  &&= (3 \delta^2 +1) B_1 -2 \delta B_2 - B_3, \\
  \ p_2(L_1,L_2,L_3,L_4) &&= -2 \delta^2 C_1-2  \delta C_3+ \delta C_4-C_7+C_8
 = -2\delta^2 B_1 +\delta B_2  +B_3-B_4, \\
    \ e_3(L_1,L_2,L_3,L_4) &&=
(4 \delta^3+4 \delta) C_1+ (3 \delta^2+1) C_2+(6 \delta^2+1) C_3+(-3 \delta^2-1) C_4+2 \delta C_5 \\
 && \quad +3 \delta C_6-3 \delta C_8+C_9 \\
 &&= (4 \delta^3 + 4 \delta) B_1 +(-3\delta^2-1) B_2 -3 \delta B_3 + B_5,\\
 \ p_3(L_1,L_2,L_3,L_4) &&=(2 \delta^3 +3 \delta) C_1 +C_2+(3 \delta^2+1) C_3 +(-\delta^2-2) C_4+\delta C_7-2 \delta C_8+C_9+C_{10}  \\
 &&=(2\delta^3+3 \delta) B_1 + (-\delta^2-2)  B_2  -  2 \delta B_3  +\delta B_4 + B_5 +B_6.
\eneqn
Observe that the matrix
\eqn      	      	
\left(\begin{array}{rrrrrr}
1 & 2 \delta & 3 \delta^{2} + 1 & -2 \delta^{2} & 4 \delta^{3} + 4 \delta & 2 \delta^{3} + 3 \delta \\
0 & -1 & -2 \delta & \delta & -3 \delta^{2} - 1 & -\delta^{2} - 2 \\
0 & 0 & -1 & 1 & -3 \delta & -2 \delta \\
0 & 0 & 0 & -1 & 0 & \delta \\
0 & 0 & 0 & 0 & 1 & 1 \\
0 & 0 & 0 & 0 & 0 & 1
\end{array}\right)
\eneqn
is invertible for every $\delta \in \C$.
Hence the evaluation of $\{p_0, p_1,p_2,p_3,e_2,e_3\}$ at $L_1,L_2,L_3,L_4$ becomes a basis of the center
$Z(B_{2,2}(\delta))$ for every  $\delta \in \C$, and hence the supersymmetric polynomials in $L_1,L_2,L_3,L_4$ generate the center.

\end{example}

\vskip1em

\section{Center of the quantized walled Brauer algebra}
In this section,
we  establish an analogue of Theorem \ref{thm:main}
for the quantized walled Brauer algebra $H_{r,s}(q,\rho)$.
The following definition appeared in \cite{KM1, Leduc}.
\begin{definition} \label{def:quantum walled Brauer algebra}
{\rm
Let $r$ and $s$ be nonnegative integers. Let $R$ be an integral domain and let $q,\rho$ be elements in $R$ such that $q^{-1}, \rho^{-1}$ and $\delta := \frac{\rho-\rho^{-1}}{q-q^{-1}}$ lie in $R$.
We denote by $H_{r,s}^{R}(q,\rho)$  the associative algebra with $\one$  over $R$ generated by
 $S_1, \ldots, S_{r-1},$ $S_{r+1}, \ldots,$ $S_{r+s-1}$,
  $E_{r,r+1}$ with the defining relations
\begin{equation*}
 \begin{aligned}
    &(S_i-q)(S_i+q^{-1})=0, \quad  S_iS_{i+1}S_i=S_{i+1}S_iS_{i+1}, \quad  S_iS_j=S_jS_i \ \  (|i-j|>1), \\
 & E_{r,r+1}^2= \delta E_{r,r+1}, \quad  E_{r,r+1} S_j =S_j E_{r,r+1} \ \ (j \neq r-1,r+1),\\
 & \rho E_{r,r+1}=E_{r,r+1} S_{r-1} E_{r,r+1} =E_{r,r+1} S_{r+1} E_{r,r+1}, \\
& E_{r,r+1} S_{r-1}^{-1} S_{r+1} E_{r,r+1} S_{r-1}   = E_{r,r+1} S_{r-1}^{-1} S_{r+1} E_{r,r+1} S_{r+1}, \\
& S_{r-1} E_{r,r+1} S_{r-1}^{-1} S_{r+1} E_{r,r+1} = S_{r+1} E_{r,r+1} S_{r-1}^{-1} S_{r+1} E_{r,r+1}.
  \end{aligned}
\end{equation*}
}
\end{definition}
Note that the element $S_i^{-1}$ in the last two relations is given by $S_i^{-1}= S_i-(q-q^{-1}) \one$ from the relation in the first line.

By
\cite{Leduc},
the algebra $H_{r,s}^{R}(q,\rho)$ has an $R$-linear basis
 which is in bijection with  $(r,s)$-walled Brauer diagrams:
 for each $(r,s)$-walled Brauer diagram $c$, we attach a monomial  $T_c \in H_{r,s}^R(q,\rho)$ in $S_i, S_i^{-1}, E_{r,r+1}$  so that the expression of each basis element does not depend on the choice of the base ring $R$. 
See \cite{Leduc} for the explicit expression.
It is called the \emph{standard monomial basis}. See also \cite{DDS, Halverson, KM2}.

Now assume that we have a ring homomorphism $\phi:R \to S$ for some integral domain $S$.
Then $H_{r,s}^S(\phi(q),\phi(\rho))$ becomes an $R$-algebra via $\phi$, and we have an $R$-algebra homomorphism
$\tilde \phi : H_{r,s}^R(q,\rho) \to H_{r,s}^S(\phi(q),\phi(\rho))$,
sending the generators $E_{r, r+1}$ and $S_i$'s to themselves.
By the definition of standard monomial basis, the homomorphism $\widetilde \phi$ sends $T_c$ to $T_c$ for each $(r,s)$-walled Brauer diagram $c$.
Note that $\{T_c \} \subset H_{r,s}^S(\phi(q),\phi(\rho))$ is a  linearly independent subset over $R$ if and only if $\phi$ is injective.
Hence, if $\phi$ is injective, then $\tilde \phi$ is also injective and the
image of $\widetilde \phi$ is isomorphic to $H_{r,s}^R(q,\rho)$ as an $R$-algebra.
In particular, the image of $\widetilde \phi$ has an $R$-basis $\{T_c\}$ so that it can be characterized as the
$R$-subalgebra of $H_{r,s}^S(\phi(q),\phi(\rho))$ generated by $E_{r,r+1}, S_i \ (i=1,\ldots, r-1,r+1, \ldots, r+s-1)$.

The following definition is motivated by Definition \ref{def:JM element} as well as \cite{Morton}.
\begin{definition} \label{def:quantum JM}
{\rm
Set
\eqn
\mathcal L_1: = 0, \quad \mathcal L_{r+1} := \rho\Big(\sum_{j=1}^{r} -E_{j,r+1} +\delta \Big) \  \in \ H^R_{r,s}(q, \rho),
 \eneqn
 where
 \eqn
 E_{j,k}:=(S_{k-1} \cdots S_{r+1} )(S_j^{-1} \cdots S_{r-1}^{-1} )E_{r,r+1}  (S_{r-1}^{-1} \cdots S_j^{-1})(  S_{r+1} \cdots S_{k-1} ) \quad \text{for } \ j \le r < k.
 \eneqn
 Then we define
\begin{eqnarray} \label{def:quantum JM element}
\mathcal L_k
:=
\begin{cases}
S_{k-1}^{-1} \mathcal L_{k-1} S_{k-1}^{-1}+ S_{k-1}^{-1} & \text{if } \ 2\le k \le r, \\
S_{k-1} \mathcal L_{k-1} S_{k-1} +S_{k-1}  & \text{if }\  r+2 \le k \le r+s.
\end{cases}
\end{eqnarray}

We call these $\mathcal L_k$'s the \emph{Jucys-Murphy elements of $H^R_{r,s}(q,\rho)$}
}.
\end{definition}

For convenience, we set
\eqn
&& T_{(a,b)}=T_{(b,a)}:= \begin{cases} S_{b-1}^{-1} \cdots S_{a+1}^{-1} S_a^{-1} S_{a+1}^{-1}  \cdots S_{b-1}^{-1}
 \ \text{for } 1 \le a < b \le r, \\
                      S_{b-1} \cdots S_{a+1} S_a S_{a+1}  \cdots S_{b-1}
 \ \  \text{   for }   r+1 \le a < b \le r+s. \end{cases}
\eneqn

Calculating $\mathcal{L}_k$, we have
\eqn
\mathcal L_k=
\begin{cases}
\sum_{j=1}^{k-1} T_{(j,k)}     & \text{if } \ 2\le k \le r, \\
\rho \big(\sum_{j=1}^{r} -E_{j,k} +\delta \big) + \rho^2 \sum_{j=r+1}^{k-1} T_{(j,k)}  & \text{if }\  r+2 \le k \le r+s.
\end{cases}
\eneqn

\begin{remark}
  The elements $\mathcal L_1, \ldots, \mathcal L_r$ can be considered as Jucys-Murphy elements of Hecke algebra $H_r^R(q)$, which were firstly introduced in \cite{DJ}.
  Indeed, 
if we replace $T_{(i,i+1)}$ and $q$ in \cite{DJ} with $q^{-1}S_i^{-1}$ and $q^{-2}$, respectively,
  we get $\widetilde{L}_i = q \mathcal L_i$ ($i\ge 2$), where $\widetilde{L}_i$ is the image of \emph{Murphy operator} called in \cite{DJ}.
\end{remark}

Proposition \ref{prop:quantum commuting relations with generators} is analogous to Proposition \ref{prop:commuting relations with generators}.
To prove it, we need a lemma.
\begin{lemma} \label{lem:quantum relations between horizontal lines and vertical lines}
For all admissible $i, j, k$,  we have
  \begin{enumerate}[\rm(1)]
\item $S_{i}  T_{(j,k)} =T_{(j,k)} S_i$ and \ $S_i E_{j,k}=E_{j,k} S_i$ for $i, i+1 \neq j,k$,
\item $ S_i E_{i,k} =E_{i+1,k} S_i - (q-q^{-1}) E_{i+1,k}$ and \  $E_{i,k} S_i =S_i E_{i+1,k} -(q-q^{-1}) E_{i+1,k}$,
\item $E_{r,r+1} E_{j,k} = E_{j,k} E_{r,r+1} $ for $j \neq r, k \neq r+1$,
\item $E_{r,r+1} E_{i,r+1} = \rho^{-1} E_{r,r+1} T_{(i,r)}$ and \ $E_{i,r+1} E_{r,r+1} = \rho^{-1} T_{(i,r)}E_{r,r+1}$ for $i \neq r$,\\
        $E_{r,r+1} E_{r,k} \ \  = \rho \; E_{r,r+1} T_{(r+1,k)}$ and \ $E_{r,k} E_{r,r+1} \ \ =  \rho \;  T_{(r+1,k)}E_{r,r+1}$ for $k \neq r+1$.
\end{enumerate}
\end{lemma}

\Proof (1)The first relation can be checked from the defining relations
$S_i S_{i+1} S_i =S_{i+1} S_i S_{i+1}$ and $S_i S_j = S_j S_i  \ \ \ (|i-j|>1)$. 
For the second relation, it is trivial when $i \le j-2$ or $k+1 \le i$.
Let $j+1 \le i <r$. Note that $S_{i-1} E_{r,k}=E_{r,k}S_{i-1}$. We have
\eqn
S_i^{-1} E_{j,k} && =S_i^{-1} S_j^{-1} \cdots  S_{r-1}^{-1} E_{r, k} S_{r-1}^{-1} \cdots S_j^{-1} \\
&& = S_j^{-1} \cdots S_i^{-1}S_{i-1}^{-1}S_i^{-1} \cdots  S_{r-1}^{-1} E_{r, k} S_{r-1}^{-1} \cdots S_j^{-1} \\
&&=S_j^{-1} \cdots S_{i-1}^{-1}S_{i}^{-1}S_{i-1}^{-1} \cdots  S_{r-1}^{-1} E_{r, k} S_{r-1}^{-1} \cdots S_j^{-1} \\
&&=S_j^{-1} \cdots  S_{r-1}^{-1} E_{r, k} S_{r-1}^{-1} \cdots  S_{i-1}^{-1} S_{i}^{-1} S_{i-1}^{-1} \cdots S_j^{-1} \\
&&=S_j^{-1} \cdots  S_{r-1}^{-1} E_{r, k} S_{r-1}^{-1} \cdots  S_{i}^{-1} S_{i-1}^{-1} S_{i}^{-1} \cdots S_j^{-1} \\
&& =S_j^{-1} \cdots  S_{r-1}^{-1} E_{r, k} S_{r-1}^{-1}  \cdots S_j^{-1} S_i^{-1} \\
&&=E_{j,k} S_i^{-1},
\eneqn
hence $E_{j,k} S_i =S_i E_{j,k}$. Similarly, we can obtain $S_i E_{j,k}=E_{j,k} S_i$ for $r< i \le k-2$.

(2) By the definition of $E_{i,k}$, we get
\eqn
S_i E_{i,k}= E_{i+1, k} S_i^{-1}=E_{i+1, k} (S_i -(q-q^{-1})\one)=E_{i+1, k}S_i-(q-q^{-1})E_{i+1, k}.
\eneqn
Similarly, we can obtain another relation.

(3) From the last two defining relations of $H_{r,s}^R(q,\rho)$, we get
$$E_{r,r+1} S_{r-1}^{-1} S_{r+1} E_{r,r+1} S_{r+1}S_{r-1}^{-1} = S_{r-1}^{-1}S_{r+1} E_{r,r+1} S_{r-1}^{-1} S_{r+1} E_{r,r+1}, $$
that is, $E_{r,r+1} E_{r-1, r+2} =E_{r-1,r+2} E_{r,r+1}.$
Using it and relations $S_i E_{r,r+1}=E_{r,r+1}S_i$ ($i \neq r-1, r+1$), we get the desired relations.

(4) From the defining relations of $H_{r,s}^R(q,\rho)$, we have
$\rho^{-1} E_{r,r+1} =E_{r,r+1} S_{r-1}^{-1} E_{r,r+1}. $
Using it and  relations $S_i E_{r,r+1}=E_{r,r+1}S_i$ ($i \neq r-1, r+1$), we  obtain the first relation as follows:
\eqn
E_{r,r+1} E_{i,r+1} && =E_{r,r+1} S_{i}^{-1} \cdots S_{r-1}^{-1} E_{r,r+1} S_{r-1}^{-1} \cdots S_i^{-1} \\
&& = S_{i}^{-1} \cdots S_{r-2}^{-1} E_{r,r+1} S_{r-1}^{-1}E_{r,r+1} S_{r-1}^{-1} \cdots S_i^{-1} \\
&& = \rho^{-1} S_{i}^{-1} \cdots S_{r-2}^{-1} E_{r,r+1} S_{r-1}^{-1} \cdots S_i^{-1} \\
&& = \rho^{-1} E_{r,r+1} T_{(i,r)}.
\eneqn
Similarly we can get the second relation.  Using $\rho E_{r,r+1} =E_{r,r+1} S_{r+1} E_{r,r+1}$, we obtain the third and the fourth relation.
\QED

\begin{prop} \label{prop:quantum commuting relations with generators}
We have the following relations:
\begin{enumerate}[\rm(1)]
\item  $S_i \mathcal L_{i+1} -\mathcal L_i S_i=1 -(q-q^{-1}) \mathcal{L}_i$ \ \ \ for $ \ i<r$, \\
  $\mathcal L_{i+1} S_i-S_i \mathcal L_i=1-(q-q^{-1}) \mathcal{L}_i$ \ \ \ for $\ i <r$,
 \item         $S_i \mathcal L_{i+1} -\mathcal L_i S_i=1 +(q-q^{-1}) \mathcal{L}_{i+1}$ for  $ \ i > r$,\\
        $\mathcal L_{i+1} S_i-S_i \mathcal L_i=1+(q-q^{-1}) \mathcal{L}_{i+1}$ for $ \  i >r$,
  \item $E_{r, r+1} (\mathcal L_r +\mathcal L_{r+1}) =(\mathcal L_r +\mathcal L_{r+1}) E_{r, r+1} =0,$
  \item $S_i \mathcal L_k = \mathcal L_k S_i$ for $k \neq i,i+1$,
  \item $E_{r, r+1} \mathcal L_k =\mathcal L_k E_{r,r+1}$ for $k \neq r, r+1$.
  \end{enumerate}
\end{prop}

\Proof
(1),(2) The relations can be checked from the definition of $\mathcal L_i$ in \eqref{def:quantum JM element} and the defining relations  $(S_i-q)(S_i+q^{-1})=0$. The relations in (1) can be also obtained from \cite[Lemma 2.3(ii),(iii)]{DJ}.

(3) From Lemma \ref{lem:quantum relations between horizontal lines and vertical lines} (4) and $E_{r,r+1}^2=\delta E_{r,r+1}$, we get
\eqn
E_{r,r+1}(\mathcal L_r + \mathcal L_{r+1}) && = E_{r,r+1} \big(\sum_{j=1}^{r-1} T_{(j,r)} + \rho \sum_{j=1}^r -E_{j,r+1} +\rho \delta \big)\\
&&= \sum_{j=1}^{r-1} E_{r,r+1}T_{(j,r)} - E_{r,r+1}T_{(j,r)}-\rho \delta E_{r,r+1} +\rho \delta E_{r,r+1}=0.
\eneqn
By similar manner, we can obtain the second relation.

(4) The case $i \ge k+1$ is trivial. Let $i < k-1 < k \le r$. Then one can check that
\begin{eqnarray} \label{eq:S_i and T_{(i,k)}}
S_i (T_{(i,k)}+T_{(i+1,k)}) =(T_{(i,k)}+T_{(i+1,k)})S_i.
\end{eqnarray}
Here, we use the relations
$$ S_i T_{(i,k)} =T_{(i+1,k)} S_i - (q-q^{-1}) T_{(i+1,k)}, \ T_{(i,k)} S_i =S_i T_{(i+1,k)} -(q-q^{-1}) T_{(i+1,k)},$$
which can be obtained from the defining relations of $H_{r,s}^R(q,\rho)$.
Using \eqref{eq:S_i and T_{(i,k)}} and Lemma \ref{lem:quantum relations between horizontal lines and vertical lines} (1),
we obtain $S_i \mathcal{L}_k =\mathcal{L}_k S_i$.

Let $i < r < k$.
By Lemma \ref{lem:quantum relations between horizontal lines and vertical lines} (2),
we have $S_i(E_{i,k} +E_{i+1,k}) =(E_{i,k} +E_{i+1,k})S_i.$
From it and Lemma \ref{lem:quantum relations between horizontal lines and vertical lines} (1),
we obtain the desired result.
For $r<i<k-1$, we can get the same relation as \eqref{eq:S_i and T_{(i,k)}}
from the relations
$$S_i T_{(i,k)}=T_{(i+1,k)}S_i + (q-q^{-1}) T_{(i,k)}, \  T_{(i,k)}S_i = S_i T_{(i+1,k)} +(q-q^{-1}) T_{(i,k)}.$$
From it and Lemma \ref{lem:quantum relations between horizontal lines and vertical lines} (1),
we can obtain $S_i \mathcal{L}_k =\mathcal{L}_k S_i$, as desired.

(5) If $k <r$, it is trivial.
Let $k \ge r+2$.
From Lemma \ref{lem:quantum relations between horizontal lines and vertical lines} (1),(3) and (4),
we can obtain $E_{r,r+1}\mathcal L_k =E_{r,r+1} \mathcal L_k$.
\QED

\begin{corollary} \label{cor:belong to center for quantum case}
The elements $\mathcal L_k$'s are commuting  with each other and
every supersymmetric polynomials in $\mathcal L_1, \ldots, \mathcal L_{r+s}$ belong to the center
$Z(H_{r,s}^R(q,\rho))$ of $H_{r,s}^R(q,\rho)$.
\end{corollary}

\Proof
Replacing $(i,i+1)=s_i, e_{r,r+1}$ with $S_i, E_{r,r+1}$, respectively, we can obtain the results by the same proofs as for Proposition \ref{prop:Li are commuting} and \ref{prop:belongtocenter}.
\QED

\begin{remark} \label{rem:quantum JM element}
 (1) The fact that $\mathcal{L}_1 + \cdots + \mathcal{L}_{r+s}$ belongs to $Z(H_{r,s}^R(q,\rho))$ was also shown in \cite[Proposition 2.5]{RS}.
  Indeed, the element $c_{r,s}$ in \cite[Proposition 2.5]{RS} is $-\rho^{-1}(\mathcal{L}_1 + \cdots + \mathcal{L}_{r+s}) +s \delta$. 
(2) Set $\Lambda:=\Z[q^{\pm1},\rho^{\pm1}]_{(q-q^{-1})}$,  the localization of the Laurent  polynomial ring in $q$ and  $\rho$ at $(q-q^{-1})$.
  Then the  above corollary was shown in \cite[Theorem 1]{Morton} by a skein theoretic description of $H_{r,s}^\Lambda (q,\rho)$.
Let us explain it  more precisely.
Hugh Morton studied  a relation between the framed HOMFLY skein module on the annulus and 
the center of the framed HOMFLY skein module on the rectangle with designated inputs and outputs boundary points. 
The latter skein module with a naturally defined multiplication is isomorphic to $H_{r,s}^\Lambda (q,\rho)$,
where the parameters $s$ and $v$ used there correspond to $q^{-1}$ and $\rho$, respectively.
The diagrams corresponding to the generators are:

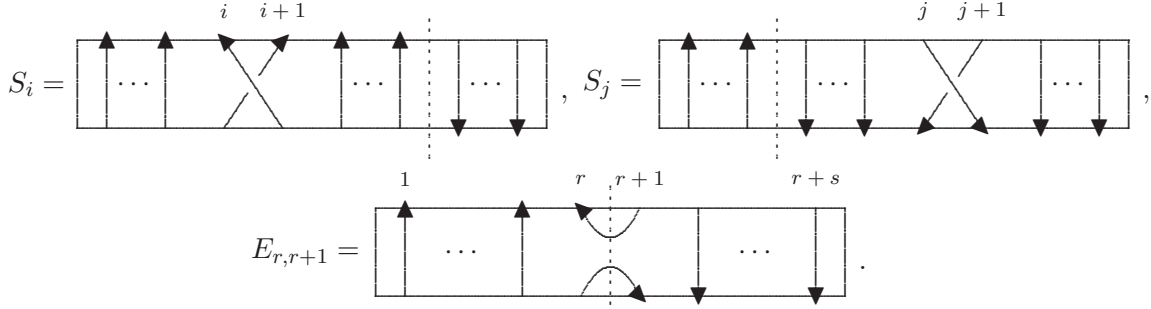
\newcommand{\ul}{\mathbin{\rotatebox[origin=c]{45}{$\blacktriangle$}}}
\newcommand{\ur}{\mathbin{\rotatebox[origin=c]{315}{$\blacktriangle$}}}
\newcommand{\dl}{\mathbin{\rotatebox[origin=c]{135}{$\blacktriangle$}}}
\newcommand{\dr}{\mathbin{\rotatebox[origin=c]{225}{$\blacktriangle$}}}
${\beginpicture
\setcoordinatesystem units <0.78cm,0.39cm>
\setplotarea x from 0 to 9, y from -1 to 4
\put{$S_i= $} at -0.2 1.5

\plot 0.5 0 0.5 3 /
\plot 0.5 3 8.5 3 /
\plot 8.5 3 8.5 0 /
\plot 8.5 0 0.5 0 /

\put{} at  1 0  \put{$\blacktriangle$} at  1 3
\put{} at  2 0  \put{$\blacktriangle$} at  2 3
\put{} at  3 0  \put{$\ul$} at  3 3
\put{} at  4 0  \put{$\ur$} at  4 3
\put{$\blacktriangle$} at  5 3  
\put{$\blacktriangle$} at  6 3  
\put{$\blacktriangledown$} at  7 0  
\put{$\blacktriangledown$} at  8 0  
\put{$\cdots$} at 1.5 1.5
\put{$\cdots$} at 5.5 1.5
\put{$\cdots$} at 7.5 1.5
\put{{\scriptsize$i$}} at 3 4
\put{{\scriptsize$i+1$}} at 4 4
\put{,} at 8.8 1
\plot 1 3 1 0 /
\plot 2 3 2 0 /
\plot 3 3 4 0 /
\plot 3 0 3.4 1.2 /
\plot  3.6 1.8 4 3 /
\plot 5 3 5 0 /
\plot 6 3 6 0 /
\plot 7 3 7 0 /
\plot 8 3 8 0 /
\setdashes  <.4mm,1mm>
\plot 6.5 -1   6.5 4 /
\setsolid
\endpicture}$
${\beginpicture \setcoordinatesystem units <0.78cm,0.39cm>
\setplotarea x from 0 to 9, y from -1 to 4
\put{$S_j= $} at -0.3 1.5

\plot 0.5 0 0.5 3 /
\plot 0.5 3 8.5 3 /
\plot 8.5 3 8.5 0 /
\plot 8.5 0 0.5 0 /

\put{$\blacktriangle$} at  1 3
 \put{$\blacktriangle$} at  2 3
 \put{$\blacktriangledown$} at  3 0
\put{$\blacktriangledown$} at  4 0
\put{$\dl$} at  5 0
\put{$\dr$} at 6 0
 \put{$\blacktriangledown$} at  7 0
\put{$\blacktriangledown$} at  8 0

\put{$\cdots$} at 1.5 1.5
\put{$\cdots$} at 3.5 1.5
\put{$\cdots$} at 7.5 1.5
\put{,} at 8.8 1
\put{{\scriptsize$j$}} at 5 4
\put{{\scriptsize$j+1$}} at 6 4

\plot 1 3 1 0 /
\plot 2 3 2 0 /
\plot 3 3 3 0 /
\plot 4 3 4 0 /
\plot 5 3 6 0 /
\plot 7 3 7 0 /
\plot 8 3 8 0 /

\plot 5 0 5.4  1.2 /
\plot 5.6 1.8  6 3 /

\setdashes  <.4mm,1mm>
\plot 2.5 -1   2.5 4 /
\setsolid
\endpicture}$

\begin{center}
${\beginpicture
\setcoordinatesystem units <0.78cm,0.39cm>
\put{$E_{r,r+1} = $} at -0.7 1.5
\plot 0.5 0 0.5 3 /
\plot 0.5 3 8.5 3 /
\plot 8.5 3 8.5 0 /
\plot 8.5 0 0.5 0 /
 \put{$\blacktriangle$} at  1 3
\put{$\blacktriangle$} at  3 3
\put{$\ul$} at  4 3
\put{$\dr$} at  5 0 
\put{$\blacktriangledown$} at  6 0  
\put{$\blacktriangledown$} at  8 0  

\put{$\cdots$} at 2 1.5
\put{$\cdots$} at 7 1.5
\put{.} at 8.8 1
\put{{\scriptsize $1$}} at 1 4
\put{{\scriptsize $r$}} at 4 4
\put{{\scriptsize $r+1$}} at 5 4
\put{{\scriptsize $r+s$}} at 8 4
\plot 1 3 1 0 /
\plot 3 3 3 0 /
\plot 8 3 8 0 /
\plot 6 3 6 0 /
\setdashes  <.4mm,1mm>
\plot 4.5 -1   4.5 4 /
\setsolid
\setquadratic
\plot 4 3 4.5 2 5 3 /
\plot 4 0 4.5 1 5 0  /
\endpicture}$
\end{center}
Morton  introduced  certain elements $T(j), U(k)$, which have the relations with $\mathcal L_i$'s as follows: 
\eqn
&&T(j)  =  -(q-q^{-1}) \mathcal L_j+1 \quad (1\le j\le r), \\
&&U(k)  =  \rho^{-2} (q-q^{-1}) \mathcal L_{ r+k } + \rho^{-2} \quad (1 \le k \le s ).
\eneqn
The elements $T(j)$ and $U(k)$, which is called \emph{Murphy operators} in \cite{Morton}, correspond to simple braid diagrams in the skein theoretic description. See  \cite[Section 3]{Morton} for the diagrammatic presentation of $T(j)$ and $U(k)$.
It was shown that the elements of the form
$$\sum_{j=1}^r (\rho^{-1} \, T(j))^m - \sum_{k=1}^s (\rho \, U(k))^m \quad (m \in \Z_{\ge0})$$
belong to the center of $H_{r,s}^\Lambda (q,\rho)$. Moreover it was conjectured that the above elements generate the center of $H_{r,s}^\Lambda (q,\rho)$.
Note that it is equivalent to saying that the supersymmetric polynomials in $\mathcal L_i$'s  generate the center of $H_{r,s}^\Lambda (q,\rho)$.

 (3)  Recently A. M. Semikhatov and I. Y. Tipunin introduced  some elements $\{J(r)_i\, ;\, 1 \le i \le r+s\}$ having similar diagrammatic presentations to $T(j)$ and $U(k)$, which are also called \emph{Jucys-Murphy elements} in \cite[Section 2.4]{ST}. In particular, $J(r)_{r+k}$ coincides with $U(k)$ for $1\le k \le s$ under an isomorphism between  their algebra $\mathsf{qw} \mathcal B_{r,s}$ and $H_{r,s}^{\C(q,\rho)}(q,\rho)$.
\end{remark}

We set $H_{r,s}(q,\rho):=H_{r,s}^{\C(q,\rho)}(q,\rho)$ and
$H_{r,s}(N):=H_{r,s}^{\C(q)}(q,q^N)$ for $N \in \Z_\ge0$, respectively.
The latter algebra appeared in  \cite{KM2}   for the first time.
Note that $H_{r,s}^{\C(q,\rho)}(q,\rho)$ is semisimple by, for example, \cite[Theorem 6.10]{RS}. 
We  call these two families  the \emph{quantized walled Brauer algebras}.
For the representation theory of these algebras, see, for example, \cite{DDS, Enyang, RS} and references therein.
In particular, the set $\dot{\Lambda}_{r,s}$ parametrizes the isomorphism classes of simple modules of $H_{r,s}(q,\rho)$ and $H_{r,s}(N)$ as well as the ones of $B_{r,s}(\delta)$.
Thus we know that
$$\dim_{\C(q,\rho)}Z(H_{r,s}(q,\rho))  =\dim_{\C}Z(B_{r,s}(\delta)),$$
where $\delta$ is generic.
If  $H_{r,s}(N)$ are semisimple, then
$$\dim_{\C}Z(B_{r,s}(N))= \dim_{\C(q)}Z(H_{r,s}(N)),$$
since $B_{r,s}(N)$ is also semisimple by  comparing  \cite[Theorem 6.10]{RS} with Proposition \ref{prop:semisimplicity}.

\begin{theorem} \label{thm:main quantum}
Suppose that  $H_{r,s}(N)$ are semisimple.
Then the center $Z(H_{r,s}(N))$
of $H_{r,s}(N)$ is generated by the supersymmetric polynomials in the Jucys-Murphy elements $\mathcal L_1, \ldots, \mathcal L_{r+s}$ with coefficients in $\C(q)$.

The center $Z(H_{r,s}(q,\rho))$  of $H_{r,s}(q,\rho)$ is also generated by the supersymmetric polynomials in the Jucys-Murphy elements $\mathcal L_1, \ldots, \mathcal L_{r+s}$ with coefficients in $\C(q,\rho)$.
\end{theorem}

\begin{proof}
Set $d=\dim Z(B_{r,s}(N))=|\dot\Lambda_{r,s}|$.
Let $$\set{P_i \in S_{r,s}[x;y]}{1\le i \le d}$$ be a set of supersymmetric polynomials such that
$\set{P_i(L_1,\ldots,L_{r+s}) \in B_{r,s}(N)}{1\le i \le d}$ forms a basis of $Z(B_{r,s}(N))$.
We will show first that
$$\set{P_i(\mathcal L_1,\ldots,\mathcal L_{r+s}) \in H_{r,s}(N)}{1\le i \le d}$$
is a linearly independent subset of $H_{r,s}(N)$ over $\C(q)$.
Assume that there exists
 a nontrival $\C(q)$-linear relation
$$\sum_{i=1}^d a_i(q) P_i(\mathcal L_1,\ldots,\mathcal L_{r+s})=0.$$
Then by multiplying the least common multiple of the denominators, we may assume that
$a_i(q) \in \C[q]$ for all $i$.
Dividing the above by a suitable power of $q-1$, we may further assume that
there exists $i_0$ such that $a_{i_0}(1)\neq 0$.
Now let us denote by $H_{\C[q^{\pm 1}]}$ the $\C[q^{\pm1}]$-subalgebra of $H_{r,s}(N)$ generated by $E_{r,r+1}$ and $S_i$'s. Recall that $H_{\C[q^{\pm 1}]}$ is isomorphic to $H_{r,s}^{\C[q^{\pm1}]}(q, q^N)$.
Then we have a ring homomorphism
$\phi_{q=1} : \C[q^{\pm1}] \to \C$ and $\C[q^{\pm1}]$-algebra homomorphism
$$
\widetilde \phi_{q=1} : H_{\C[q^{\pm1}]} \to B_{r,s}(N), \quad
E_{r,r+1} \mapsto E_{r,r+1},  \ S_i \mapsto (i,i+1),
$$
which can be justified by checking the defining relations. 
For example, we have
$$\widetilde \phi_{q=1}(E_{r,r+1}^2-\delta E_{r,r+1}) = E_{r,r+1}^2-\phi_{q=1}\Big(\frac{q^N-q^{-N}}{q-q^{-1}}\Big) E_{r,r+1}=E_{r,r+1}^2-N E_{r,r+1}=0.$$
Note that
$\mathcal L_k \in H_{\C[q^{\pm1}]}$  and
$\phi_{q=1}(\mathcal L_k)=L_k$ for all $1\le k\le r+s$.
Hence we have
\begin{equation*}
\begin{aligned}
0&=\widetilde\phi_{q=1}\Big(\sum_{i=1}^d a_i(q) P_i(\mathcal L_1,\ldots,\mathcal L_{r+s}) \Big)
=\sum_{i=1}^d a_i(1)  P_i \big(L_1,\ldots,L_{r+s} \big).
\end{aligned}
\end{equation*}
It follows that  $a_i(1)=0$ for all $1\le i\le d$, which yields a contradiction since
$a_{i_0}(1)\neq 0$.
Thus $\set{P_i(\mathcal L_1,\ldots,\mathcal L_{r+s}) \in H_{r,s}(N)}{1\le i \le d}$ is linearly independent over $\C(q)$, as desired.

Next, we consider the set
\begin{equation} \label{eq:basis}
\set{P_i(\mathcal L_1,\ldots,\mathcal L_{r+s}) \in H_{r,s}(q,\rho)}{1\le i \le d}
\end{equation}
and assume that
there exists 
a nontrivial $\C(q,\rho)$-linear relation
\begin{equation} \label{eq:linearly dependent}
\sum_{i=1}^d a_i(q,\rho) P_i(\mathcal L_1,\ldots,\mathcal L_{r+s})=0.
\end{equation}
By multiplying the least common multiple of the denominators and then
dividing a suitable power of $\rho-q^N$, we may assume that $a_i (q,\rho) \in \C[q,\rho]$ $(1\le i\le d)$
and there exists $i_0$ such that $a_{i_0}(q, q^N) \neq 0$.
Now consider the ring homomorphism
$\phi_{\rho=q^N} : \A \to \C(q)$ given by $q\mapsto q, \ \rho \mapsto q^N$ and
the  $\A$-algebra homomorphism given by
$$\widetilde \phi_{\rho=q^N} : H_\A \to H_{r,s}(N)
 \quad
E_{r,r+1} \mapsto E_{r,r+1},  \ S_i \mapsto S_i,
$$
where $\A=\C[q^{\pm1}, \rho^{\pm1}, \frac{\rho-\rho^{-1}}{q-q^{-1}}] \subset \C(q,\rho)$ and
$H_\A$ is the $\A$-subalgebra of $H_{r,s}(q,\rho)$ generated by $E_{r,r+1}$ and $S_i$'s.
Since
$\mathcal L_k \in H_{\A}$  and
$\phi_{q=\rho^N}(\mathcal L_k)=\mathcal L_k$ for all $1\le k\le r+s$, applying
$\widetilde \phi_{\rho=q^N} $ to \eqref{eq:linearly dependent} yields a contradiction to the facts that  $a_{i_0}(q,q^N)\neq0$, similarly as before.
Hence the set  \eqref{eq:basis} is linearly independent over $\C(q,\rho)$, as desired.
\end{proof}

\vskip1em

\begin{thebibliography}{10}

\bibitem{BCHLLS}
G. Benkart, M. Chakrabarti, T. Halverson, R. Leduc, C. Lee and J. Stroomer,
\emph{Tensor product
representations of general linear groups and their connections with Brauer algebrs},
J. Algebra \textbf{166} (1994), 529--567. \MR{1280591 (95d:20071)}



\bibitem{BS}
J. Brundan and C. Stroppel, \emph{Gradings on walled {B}rauer
  algebras and {K}hovanov's arc algebra}, Adv. Math. \textbf{231} (2012),
  no.~2, 709--773. \MR{2955190}

\bibitem{CDDM}
A. Cox, M. De~Visscher, S. Doty, and P. Martin, \emph{On the blocks
  of the walled {B}rauer algebra}, J. Algebra \textbf{320} (2008), no.~1,
  169--212. \MR{2417984 (2009c:16094)}

\bibitem{CW}
 J. Comes and B. Wilson, \emph{Deligne's category
 $\underline{Rep}(GL_{\delta})$ and representations of general linear supergroups},
Represent. Theory \textbf{16} (2012), 568--609.  \MR{2998810}

\bibitem{DG}
P. Diaconis and C. Greene, \emph{Applications of Murphy's elements},
Stanford University Tech. Report, No. 335 (1989).


\bibitem{DDS}
R. Dipper, S. Doty and F. Stoll,
\emph{The quantized walled Brauer algebra and mixed tensor space},
Algebr. Represent. Theory {\bf17} (2014), no. 2, 675--701.
\MR{3181742}

\bibitem{DJ}
 R. Dipper, G. James,
\emph{Blocks and idempotents of Hecke algebras of general linear groups},
Proc. London Math. Soc. \textbf{s3-54} (1987), no. 1, 57--82.
\MR{0872250 (88m:20084)}



\bibitem{Enyang}  J. Enyang,
\emph{Cellular bases of the two-parameter version of the centraliser algebra for the mixed tensor representations of the quantum general linear group}, Combinatorial representation theory and related topics (Japanese) (Kyoto, 2002), Surikaisekikenkyusho Kokyuroku no.1310 (2003), 134-153.
\MR{2012192}.

\bibitem{Feray}
V. F{\'e}ray, \emph{On complete functions in Jucys-Murphy elements},
Ann. Comb. {\bf16} (2012), no. 4, 677--707.
\MR{3000438}

\bibitem{Halverson}
T. Halverson, \emph{Characters of the centralizer algebras of mixed tensor representations of
$GL(r,\C)$ and the quantum group $U_q(gl(r,\C))$},
Pacific J. Math. \textbf{174} (1996), no. 2, 359--410.
\MR{1405593 (97h:17017)}


\bibitem{JK}
G. James and A. Kerber, \emph{The Representation Theory of the Symmetric Group},
Encyclopedia of Mathematics and its Applications, 16. Addison-Wesley Publishing Co., Reading, Mass., 1981. xxviii+510 pp. \MR{0644144 (83k:20003)}

\bibitem{Jucys}
A.-A.~A. Jucys, \emph{Symmetric polynomials and the center of the symmetric
  group ring}, Rep. Mathematical Phys. \textbf{5} (1974), no.~1, 107--112.
  \MR{0419576 (54 \#7597)}


\bibitem{Karp}
G. Karpilovsky, \emph{The {J}acobson radical of group algebras},
  North-Holland Mathematics Studies, vol.135, North-Holland Publishing Co.,
  Amsterdam, 1987, Notas de Matem{\'a}tica [Mathematical Notes], 115.
  \MR{886889 (88g:16013)}

\bibitem{Koike}
K. Koike, \emph{On the decomposition of tensor products of the representations of classical groups: by means
of the universal characters}, Adv. Math. \textbf{74} (1989), no. 1, 57--86.
\MR{0991410 (90j:22014)}

\bibitem{KM1}  M. Kosuda and J.  Murakami, \emph{The centralizer algebras of mixed tensor representations of $U_q(\mathfrak{gl}_n)$ and the HOMFLY polynomial of links}, Proc. Japan Acad. Ser. A Math. Sci.{\bf 68} (1992), no. 6, 148--151.\MR{1179388}.

\bibitem{KM2}  M. Kosuda and J.  Murakami,
\emph{Centralizer algebras of the mixed tensor representations of quantum group
$U_q(\mathfrak{gl}(n,\mathbb C))$}, Osaka J. Math. {\bf30} (1993), no. 3, 475--507. \MR{1240008}.

\bibitem{LascouxThibon}
A. Lascoux and J.-Y. Thibon,
\emph{Vertex operators and the class algebras of symmetric groups},
J. Math. Sci. (N. Y.) \textbf{121} (2004), no. 3, 2380--2392.
\MR{1879068 (2003b:20019)}


\bibitem{Lassalle}
M. Lassalle, \emph{Class expansion of some symmetric functions in Jucys-Murphy elements},
J. Algebra {\bf394} (2013), 397--443.
\MR{3092727}

\bibitem{Leduc}
 R. Leduc,
\emph{A two-parameter version of the centralizer algebra of the mixed tensor representations of the general linear group and quantum general linear group},
Thesis (Ph.D.) The University of Wisconsin - Madison. 1994. 151 pp.
\MR{2691209}

\bibitem{Li}
Y. Li, \emph{Jucys-{M}urphy elements and centres of cellular algebras},
  Bull. Aust. Math. Soc. \textbf{85} (2012), no.~2, 261--270. \MR{2891540}

\bibitem{Mathas}
A. Mathas, \emph{Seminormal forms and Gram determinants for cellular algebras},
J. Reine Angew. Math. \textbf{619} (2008), 141--173. \MR{2414949 (2009e:16059)}

\bibitem{MatsumotoNovak}
S. Matsumoto and J. Novak, \emph{Jucys-Murphy elements and unitary matrix integrals},
Int. Math. Res. Not. IMRN (2013), no. 2, 362--397. \MR{3010693}

\bibitem{Moens}
E. Moens, \emph{Supersymmetric Schur functions and Lie superalgebra representations},
Ph.D. thesis, Ghent University, 2007.

\bibitem{Murphy81}
G.~E. Murphy, \emph{A new construction of Young's seminormal representation
 of the symmetric group}, J. Algebra \textbf{69} (1981), no.~1, 287--291.
 \MR{0617079 (82h:20014)}

\bibitem{Murphy}
G.~E. Murphy, \emph{The idempotents of the symmetric group and {N}akayama's
  conjecture}, J. Algebra \textbf{81} (1983), no.~1, 258--265. \MR{696137
  (84k:20007)}

\bibitem{Nazarov}
M. Nazarov, \emph{Young's orhogonal form for Brauer's centralizer algebra},
J. Algebra \textbf{182} (1996), no.~3, 664--693. \MR{1398116 (97m:20057)}

\bibitem{Morton}
 H. Morton, \emph{Power sums and Homfly skein theory},
 Invariants of knots and 3-manifolds (Kyoto, 2001), 235--244,
Geom. Topol. Monogr., 4, Geom. Topol. Publ., Coventry, 2002.
\MR{2002613}

\bibitem{OV2} A. Okounkov and A. Vershik, \emph{A new approach to the representation theory
of the symmetric groups.II}, J. Math. Sci. (N. Y.)
\textbf{131} (2005), no. 2, 5471--5494.  \MR{2050688 (2005c:20024)}


\bibitem{RuiSu}
H. Rui and Y. Su, \emph{Affine walled Brauer algebras and super Schur-Weyl duality},
 Adv. Math. \textbf{285} (2015), 28--71. \MR{3406495}.


\bibitem{RS}
H. Rui and L. Song, \emph{The representations of quantized walled
  {B}rauer algebras}, J. Pure Appl. Algebra \textbf{219} (2015), no.~5,
  1496--1518. \MR{3299691}


\bibitem{Sartori}
A. Sartori, \emph{The degenerate affine walled Brauer algebras},
J. Algebra \textbf{417} (2014), 198--233. \MR{3244645}

\bibitem{SartoriStroppel} A. Sartori and C. Stroppel,
\emph{Walled Brauer algebras as idempotent truncations of level 2 cyclotomic quotients},
J. Algebra \textbf{440} (2015), 602--638. \MR{3373407}

\bibitem{ST}
 A. M. Semikhatov and I. Y. Tipunin,
\emph{Quantum walled Brauer algebra: commuting families, baxterization, and representations},
J. Phys. A: Math. Theor. \textbf{50} (2017), 065202.

\bibitem{SM}
C. L. Shader and D. Moon, \emph{Mixed tensor representations and representations for the general linear
Lie superalgebras}, Comm. Algebra \textbf{30} (2002), no. 2, 839--857.  \MR{1883028 (2003a:17005)}


\bibitem{Shalile}
A. Shalile, \emph{On the center of the {B}rauer algebra}, Algebr. Represent.
  Theory \textbf{16} (2013), no.~1, 65--100. \MR{3018181}

\bibitem{Stem}
J.~R. Stembridge, \emph{A characterization of supersymmetric polynomials}, J.
  Algebra \textbf{95} (1985), no. 2, 439--444. \MR{0801279 (87a:11022)}

\bibitem{Turaev}
V. Turaev, \emph{Operator invariants of tangles, and $R$-matrices}, Math. USSR-Izv. \textbf{35} (1990), no. 2, 411--444. \MR{1024455 (91e:17011)}
\end{thebibliography}

\providecommand{\bysame}{\leavevmode\hbox to3em{\hrulefill}\thinspace}
\providecommand{\MR}{\relax\ifhmode\unskip\space\fi MR }
\providecommand{\MRhref}[2]{%
  \href{http://www.ams.org/mathscinet-getitem?mr=#1}{#2}
}
\providecommand{\href}[2]{#2}

\end{document}